\documentclass{amsart}
\usepackage[utf8]{inputenc}
\usepackage{geometry}

\usepackage{graphicx}
\usepackage{amsmath, amsfonts, amssymb,amsthm,mathtools,bm,color,mathrsfs}

\usepackage{latexsym}
\usepackage{amscd}
\usepackage[colorlinks,linkcolor=blue,anchorcolor=blue,citecolor=blue]{hyperref}
\usepackage{multirow}
\usepackage{indentfirst}
\usepackage{enumerate}
\usepackage[english]{babel}\usepackage{csquotes}
\usepackage{comment}
\usepackage{appendix}
\usepackage{url}
\numberwithin{equation}{section}
\newtheorem{theorem}{Theorem}[section] 
\newtheorem{lemma}[theorem]{Lemma}     
\newtheorem{corollary}[theorem]{Corollary}
\newtheorem{proposition}[theorem]{Proposition}
\newtheorem*{conjecture}{Conjecture}
\newtheorem*{problem*}{Problem}

\theoremstyle{definition}
\newtheorem{definition}[theorem]{Definition}

\newtheorem*{remark}{Remark}

\newcommand{\CC}{\mathbb C}

\newcommand{\PP}{\mathbb P}
\newcommand{\GG}{\mathbb G}
\newcommand{\ord}{\operatorname{ord}}
\newcommand{\exc}{\operatorname{exc}}

\title[On Vojta's general abc conjecture and Campana's orbifold conjecture ]{A complex case of Vojta's general abc conjecture and cases of Campana's orbifold conjecture}
\author{Ji Guo}
\address{School of Mathematics and Statistics, HNP-LAMA \\ Central South University \\ Changsha  410075 \\ China} \email{221250@csu.edu.cn}

\author{Julie Tzu-Yueh Wang}
\address{Institute of Mathematics\\ Academia Sinica\\6F, Astronomy-Mathematics Building\\No. 1, Sec. 4, Roosevelt Road \\ Taipei
	10617\\Taiwan} \email{jwang@math.sinica.edu.tw}

\begin{document}
	
	
	\begin{abstract}
		We showed a truncated second main theorem of level one with explicit exceptional sets for analytic maps into  $\mathbb P^2$ intersecting the coordinate lines with sufficiently high multiplicities.   The proof is based on a GCD theorem for an analytic map $f:\mathbb C\mapsto\mathbb P^n$ and two homogeneous  polynomials in $n+1$ variables with  coefficients which are meromorphic functions of  the same growth as  the analytic map $f$.		As applications, we studied some cases of Campana's orbifold conjecture   for  $\mathbb P^2$ and   finite ramified covers of  $\mathbb P^2$ with three components admitting sufficiently large multiplicities.   Moreover,  we explain how to adapt our methods to  show the strong  Green-Griffiths-Lang conjecture for  a   finite ramified covers of  $\mathbb G_m^2$.
\end{abstract}

	\thanks{2020\ {\it Mathematics Subject Classification.} Primary 30D35,  
		Secondary 11J97, 30A99.}
	\thanks{The first author is supported by National Natural Science Foundation of China NO. 12201643. The second named is supported in part by Taiwan's MoST grant 110-2115-M-001-009-MY3.}

	\baselineskip=16truept 
	\maketitle \pagestyle{myheadings}
	
	\section{Introduction}
	The Green-Griffiths-Lang conjecture in the non-compact case (see \cite[Proposition~15.3]{vojta2009diophantine}) 
	reads as follows: 
	{\it If $X$ is  a complex smooth projective variety, $D$ is a normal crossing divisor on $X$, and $X\setminus D$ is a variety of log general type, then a holomorphic map $f:\mathbb C\to X\setminus D$ cannot have Zariski-dense image.}  
	Instead of considering a holomorphic map  $f:\mathbb C\to X$ with image not intersecting the support of $D$,
	Campana took into account the  the multiplicities of $f$  intersecting the support of $D$ in \cite{Campana05}.  To formulate Campana's conjecture,   let   $X$ be a complex smooth projective variety.	
	Recall that an \textsl{orbifold divisor} \(\Delta\) is a linear combination $\sum_{Y\subset X}c_{\Delta}(Y)\cdot Y$, 
	where \(Y\) ranges over all irreducible divisors of \(X\), and
	the \textsl{orbifold coefficients} are rational numbers \(c_{\Delta}(Y)\in[0,1]\cap\mathbb Q\) such that all but finitely many are zero.
	Equivalently,
	\[
	\Delta
	=
	\sum_{\{Y\subset X\}}
	\left(
	1-m^{-1}_{\Delta}(Y)
	\right)
	\cdot
	Y,
	\]
	where only finitely   \(m_{\Delta}(Y)\in[1,\infty]\cap\mathbb Q\) are larger than \(1\).
	An orbifold pair is a pair \((X,\Delta)\), where $\Delta$ is an orbifold divisor.  The pair interpolates between the compact case where \(\Delta=\varnothing\) and the pair \((X,\varnothing)=X\) has no orbifold structure, and the {\it open}, or {\it logarithmic case} where $c_{\Delta}(D)= 1$ for all $c_{\Delta}(D)\ne 0$, and we identify \((X,\Delta)$ with $X\setminus{\rm Supp}(\Delta)\).
	Let $D_1,\hdots,D_q$ be irreducible divisors of $X$.  Let $m_1,\hdots,m_q\in(1,\infty]\cap\mathbb Q$ and 
	$\Delta=(1-m_1^{-1})D_1+\cdots+	(1-m_q^{-1})D_q$.
	Here,  consider an orbifold entire curve $f : \mathbb C \to  (X,\Delta) $, 
	i.e. an entire curve $f : \mathbb C \to X$  such that  $f(\mathbb C) \not\subset {\rm Supp}(\Delta)$
	and 
	${\rm mult}_t(f^* D_i) \ge m_i$ for all $i$ and all $t \in\mathbb C$ with $f(t) \in D_i$. Finally, we say that an orbifold pair \((X,\Delta)\) is of {\it general type} if $K_X+ \Delta$ is big,  where $K_X$ is a canonical divisor on $X$.
	
	Recall the following
	natural generalization to the orbifold category of the (strong) Green-Griffiths-Lang conjecture.  
	
	\begin{conjecture}[Campana] If  \((X,\Delta)\)  is an orbifold pair of general type, then there exists a proper closed
		subvariety $Z\subset  X $ containing the images of all nonconstant orbifold entire curves $f : \mathbb C \to  (X,\Delta).$
	\end{conjecture} 
	This conjecture has been proved by Brotbek and Deng in  \cite{BrotbekDeng2019} for $(X,\Delta)$ with $\Delta$ consisting of only one (general) component and sufficiently large multiplicity, and by Campana, Darondeau and Rousseau  in \cite{CDRousseau2020}
	for the case $X=\mathbb P^2$ and 	 $\Delta$ consisting of 11 lines with orbifold multiplicity 2.  Indeed, they both proved orbifold hyperbolicity, i.e. any orbifold entire curve  $f : \mathbb C \to  (X,\Delta) $ in their situation is constant.   It has been shown recently by Rousseau, Turchet and the second author in \cite{RTW2021} that the oribifold entire curve $f$ is algebraically degenerate for the case of smooth projective surfaces with $\Delta=\sum_{i=1}^q (1-\frac 1{m_i})D_i$, where $m_i$ is sufficiently large for each $i$ and $q\ge 4$ by the intersection criteria imposed on the $D_i$.
	
	The first major purpose of this article is to study Campana's conjecture for $\mathbb P^2$ and  its ramified covers with at least three components admitting sufficiently large multiplicities.  We now state the results in this direction.
	
	\begin{theorem}\label{orbifoldGGL}
		Let   $\Delta_0$  be an  orbifold divisor of $\mathbb P^2(\mathbb C)$   and  $H_1, H_2, H_{3}$ be three distinct lines  in $\mathbb P^2(\mathbb C)$.  Assume that the support of $\Delta_0$ and $H_1,H_2, H_{3}$ are in general position. 
		Let $m_i\in(1,\infty]\cap\mathbb Q$, $1\le i\le n$,  and  
		$\Delta=\Delta_0+(1-\frac 1{m_1})H_1+(1-\frac 1{m_2})H_2+(1-\frac 1{m_3})H_3$.  Assume that $\deg \Delta>3$.
		Then there exists a proper Zariski closed subset $W$ of $\mathbb P^2$  and  an effectively computable positive integer $\ell$  such that the image of any nonconstant  orbifold entire curve  ${\mathbf f} : \mathbb C \to  (\mathbb P^2,\Delta) $ with $\min\{m_1,m_2,m_{3}\}\ge\ell$ must be contained in $W$.	
	\end{theorem} 
	\begin{remark}\mbox{}
		\begin{enumerate}
			\item The condition that $\deg \Delta>3$ is equivalent to  that $(\mathbb P^2, \Delta)$ is of general type.
			\item The proper Zariski closed subset $W$ of $\mathbb P^2$  can be constructed explicitly.
		\end{enumerate}
	\end{remark} 
	
	\begin{theorem}\label{GG_conj}
		Let 
		$F_i$, $1\le i\le 3$, be homogeneous irreducible polynomials of positive degrees  in $\mathbb C[x_0,x_1,x_2]$.
		Assume that the plane curves  $D_i:=[F_i=0]\subset\mathbb P^2(\mathbb C)$,  $1\le i\le 3$,   intersect transversally.  
		Let   $m_i\in(1,\infty]\cap\mathbb Q$, $1\le i\le 3$ and	$\Delta= (1-\frac 1{m_1})D_1+ (1-\frac 1{m_2})D_2+(1-\frac 1{m_{3}})D_{3}$.   Suppose that  $\deg \Delta>3$.
		Then there exist  two effectively computable  positive integers $\ell$ and  $N$ such that if $\min\{m_1,m_2,m_{3}\}\ge\ell$ then  the image of any  orbifold entire curve  ${\mathbf f} : \mathbb C \to  ( \mathbb P^2,\Delta) $  is contained in a plane curve of degree bounded by $N$.
	\end{theorem}
	
	\begin{theorem}\label{finitemorphism}
		Let  $ X$ be a complex  smooth  projective  surface   of dimension $2$ with a finite  morphism $ \pi:   X\to\mathbb P^2$.  
		Let $H_i=[x_{i-1}=0]$, $1\le i\le 3$,  be the  coordinate hyperplane divisors of   $ \mathbb P^2$, and 
		$D_i$ be the support of $\pi^* H_i$ (i.e. the sum of the components of $\pi^* H_i$ counted with multiplicity 1).  Let   $m_i\in(1,\infty]\cap\mathbb Q$, $1\le i\le 3$   and   $\Delta= (1-\frac 1{m_1})D_1+(1-\frac 1{m_2})D_2+(1-\frac 1{m_{3}})D_{3}$. Let $ Z\subset X$ be the ramification divisor of $\pi$  omitting components from the support of $\Delta$.  Assume that $\pi(Z)$ does not intersect the set of points  $\{(1,0, 0),(0,1, 0),(0,0,1)\} $ in $\mathbb P^2$. 
		If the orbifold pair $(X,\Delta)$ is   of  general type, then there exist  two   positive integers $\ell$ and  $N$ such that if 
		$\min\{m_0,m_1,m_2\}\ge\ell$,  then  the image any    ${\mathbf f} : \mathbb C \to  (X,\Delta) $  is contained in an algebraic curve in $X$ with degree bounded by $N$.
	\end{theorem}
	
	\begin{remark}
		Let $D=D_1+D_2+D_{3}$.
		Recall that $X \setminus D$ is said to be   of {\it log-general type} if the divisor $D+K_X$ is big.
		It's clear that  $X \setminus D$ is of log-general type if  $(X,\Delta)$ is   of  general type.
		On the other hand, the condition that $X \setminus D$ is of log-general type implies $(X,\Delta)$ is   of  general type if each $m_i$, $1\le i\le 3$, is sufficiently large. (See \cite{LazarsfeldI} [Corollary 2.2.24].)
	\end{remark}  
	
	As noted in the beginning, the theorems above recover the corresponding results for Green-Griffiths-Lang conjecture when $m_i=\infty$ for $1\le i\le 3$.   Indeed, we can modify our proof of Theorem \ref{GG_conj} (resp. Theorem \ref{finitemorphism} ) to obtain the strong Green-Griffiths-Lang conjecture, i.e. there exists a proper Zariski closed subset $W$ of $X$ such that all  non-constant entire curves  ${\mathbf f} : \mathbb C \to X\setminus D $ are contained in $W$.
	When $X=\mathbb P^n$, the condition for $X\setminus D$ to be of log-general type is equivalent to the inequality $\deg D\ge n+2$.  When $D$ has  $n+1$ components, the Green-Griffiths-Lang conjecture is verified by Green for $n=2$ with $\deg D=4$ in \cite{green1974functional} under the assumption that ${\mathbf f}$ is of finite order, and    is solved for general $n$ by  Noguchi, Winkelman and Yamanoi  in \cite{noguchi2007degeneracy}.  Moreover,    the strong Green-Griffiths-Lang conjecture is also achieved in \cite{noguchi2007degeneracy} for   smooth surfaces of log-general type with a proper finite morphism $\pi: X\to A$, where $A$ is a semi-abelian surface.   In \cite{GSW20} and \cite{GSW22}, the case of $X=\mathbb P^n$ with $D$ consisting of  $n+1$ irreducible hypersurfaces parameterized by small functions, i.e.  moving targets of slow growth,  are studied for $\deg D= n+2$ and $\deg D\ge n+2$ respectively.

	The proofs of our main results related to Campana's orbifold conjecture  are based on the following theorem, which is of its own interest.  
	
	\begin{theorem}\label{main_thm_1}
		Let $G$ be a non-constant  homogeneous polynomial   in $\mathbb C[x_0,x_1,x_2]$ with no monomial factors and no repeated factors.  Let $H_i=[x_{i}=0]$, $0\le i\le 2$,  be the  coordinate hyperplane divisors of   $ \mathbb P^2$,
		$m_i\in(1,\infty]\cap\mathbb Q$, $0\le i\le 2$,  and   $\Delta= (1-\frac 1{m_0})H_0+(1-\frac 1{m_1})H_1+(1-\frac 1{m_{2}})H_{2}$.  Assume that the plane curve  $[G=0]$  and $H_i$, $0\le i\le 2$, are in general position.  
		Then for any $\epsilon >0$, there exists a proper Zariski closed subset $W$ and     effectively computable positive integers $\ell$  and $n$ such that for any non-constant orbifold entire curve $\mathbf{g}:\mathbb C\to  ( \mathbb P^2,\Delta)$ with $\min\{m_0,m_1,m_2\}\ge\ell$ and the image of $\mathbf{g}$ not contained in $W$,
		the following two inequalities hold.
		\begin{itemize}
			\item[\rm (i) ]  
			$N_{G(\mathbf{g})}(0,r)-N^{(1)}_{G(\mathbf{g})}(0,r)\le_{\exc} \epsilon T_{\mathbf{g}}(r)$, and 
			\item[\rm (ii) ]  
			$N^{(1)}_{G(\mathbf{g})}(0,r)\ge_{\rm exc}  (\deg  G- \epsilon)\cdot T_{\mathbf{g}}(r)$.
		\end{itemize} 
		Furthermore, the exceptional set $W$ is a finite union of closed subsets of $\mathbb P^2$ of the following type: $[x_0^{n_0}x_1^{n_1}x_2^{-n_0-n_2}=\beta]$, where $\beta\in\mathbb C$ and  $(n_0,n_1)$ is a pair of integers with $\max\{|n_0|,|n_1|\}\le n$.
	\end{theorem}
	Here,  $T_{\mathbf{g}}(r)$ is the Nevanlinna height function associated to ${\mathbf{g}}$ and $N_{G(\mathbf{g})}(0,r)$ ($N^{(1)}_{G(\mathbf{g})}(0,r)$ respectively) is the counting function associated to 0 and $G(\mathbf{g})$ (with truncation to level $1$  respectively) to be defined in the next session.  
	
	\begin{remark}
		It is clear from our proof that the exceptional set $W$ can be constructed explicitly.
	\end{remark}

	Indeed, the assertion (ii)  in Theorem \ref{main_thm_1} is a complex case of Vojta's general abc conjecture as follows.  
	(See \cite[Conjecture 15.2]{vojta2009diophantine}   and \cite[Conjecture 23.4]{vojta2009diophantine}.)
	\begin{conjecture}\label{ConjABC}
		Let $X$ be a smooth complex projective variety,   D be a normal crossing divisor on $X$,  $K_X$ be a canonical divisor on $X$, and  $A$ be an ample divisor on $X$. Then
		\begin{enumerate}
			\item[{\rm (a)}] If $f:\mathbb C\to X$ is an algebraically nondegenerate analytic map, then
			\begin{align}\label{truncate1}
				N_f^{(1)}(D,r)\ge_{\exc} T_{K_X+D,f}(r)-{\rm o}(T_{ A,f}(r)).
			\end{align}
			\item[{\rm (b)}] For any $\epsilon>0$, there exists a proper Zariski-closed subset $Z$ of $X$, depending only on $X$, $D$, $ A$, and $\epsilon$ such that  for any analytic map $f:\mathbb C\to X$ whose image is not contained in $Z$, the following
			\begin{align}\label{truncate2}
				N_f^{(1)}(D,r)\ge_{\exc} T_{K_X+D,f}(r)-\epsilon T_{ A,f}(r)
			\end{align}
			holds.
		\end{enumerate}
	\end{conjecture}
	Here, for each positive integer $n$, $N_f^{(n)}(D,r)$ is the $n$-truncated counting function with respect to $D$   given by
	\begin{align}
		N_f^{(n)}(D,r)=\sum_{0<|z|<r}\min\{{\rm ord}_z f^*D, n\}\log\frac{r}{|z|}+\min\{{\rm ord}_0 f^*D, n\}\log r,
	\end{align}
	$T_{D,f}(r)$ is the (Nenvanlinna) height function relative to the divisor $ D$ (referring to \cite[Section 12]{vojta2009diophantine}), and the notion $\leq_{\exc}$  means that the estimate holds for all $r$ outside a set of finite Lebesgue measure.
	
	If we reformulate Theorem \ref{main_thm_1} by taking $D=[G=0]+H_0+H_1+H_2$, where  $H_i:=[x_i=0]$, $0\le i\le 2$; then
	$K_{\mathbb P^2}+D$ is linearly equivalent to $[G=0]$.  Since $g_0, g_1,g_2$ are entire functions with no common zeros and sufficiently large zero multiplicity $\ell$, we can see that $N_{\mathbf{g}}^{(1)}(H_i,r)= N^{(1)}_{g_i}(0,r)\le \frac1{\ell}T_{\mathbf{g}}(r)$.  Therefore, the assertion (ii) of Theorem \ref{main_thm_1}   implies Eq.\eqref{truncate2}.   Moreover, the exceptional set $W$ in Theorem \ref{main_thm_1} can be constructed explicitly.  This also allows us to derive the strong Green-Griffiths-Lang conjecture for
	Theorem \ref{GG_conj} and  Theorem \ref{finitemorphism} with $m_i=\infty$, $1\le i\le 3$.

	There are many results in this direction with a high truncated level,  but very few with level one.  Additionally, the ability to construct an explicit exceptional set is quite limited.   The following are some known results.
	First, the conjecture holds  for $\dim X=1$.   When $X$ is a semiabelian variety, Noguchi, Winkleman and Yamanoi in \cite{noguchi2008semiabelian} showed that the inequality \eqref{truncate1} holds with $N_f^{(1)}(D,r)$ replaced by 
	$N_f^{(k_0)}(D,r)$ for some positive integer $k_0$, and \eqref{truncate2} holds if the map is algebraically nondegenerate.  
	In \cite{BrotbekDeng2019},  Brotbek and Deng also proved \eqref{truncate1} for general hypersurfaces in a smooth projective variety.
	The above conjecture is much harder for the case of moving targets, i.e. the divisor $D$ is defined over a field of ``small functions" with respect to the map $f$.  
	The only existing results   in the moving case with level one are due to Yamanoi in \cite{Yamanoi2004} for   $\dim X=1$, and in the joint work \cite {GSW22} of the two authors and Sun, where  the inequality \eqref{truncate2} is derived for complex tori with slowly growth moving targets under the assumption that  the map is multiplicatively independent over the small fields.
	
	The proof of Theorem \ref{main_thm_1} is based on the the machinery developed in \cite {GSW20}  and \cite {GSW22} for complex tori, i.e. $\mathbf{g}=(g_0,\hdots,g_n)$, where the $g_i$'s are   entire functions without zeros. It is motivated by the work of Corvaja and Zannier in  \cite{corvaja2008some}.  Consider the following example to explain the proof for Theorem \ref{main_thm_1} (i).  Let $G=x_0^2+x_1^2+x_2^2$ and $\mathbf{g}=(g_0,g_1,g_2)$, where $g_i$, $0\le i\le 2$, are  entire functions without common zeros.  Let $D_{\mathbf{g}}(G):=2\frac{g_0'}{g_0} x_0^2+2\frac{g_1'}{g_1} x_1^2+2\frac{g_2'}{g_2} x_2^2$.  Then $D_{\mathbf{g}}(G)(\mathbf{g})=G(\mathbf{g})'$ and hence $N_{G(\mathbf{g})}(0,r)-N^{(1)}_{G(\mathbf{g})}(0,r)\le N_{\rm gcd} (G(\mathbf{g}),D_{\mathbf{g}}(G)(\mathbf{g}),r).$ (See Section \ref{gcd} for definition.)  Therefore, the assertion (i) can be achieved by showing that $N_{\rm gcd} (G(\mathbf{g}),D_{\mathbf{g}}(G)(\mathbf{g}),r)\le \epsilon T_{\mathbf{g}}(r)$.  When $g_i$'s are entire functions without zeros, we have $T_{\frac{g_i'}{g_i} }(r)\le {\rm o}(T_{ \mathbf{g}}(r))$ and hence the coefficients of $D_{\mathbf{g}}(G)$ are small functions w.r.t. $\mathbf{g}$.  In this case, we can apply the GCD theorem  established by Levin and the second author in \cite{levin2019greatest}.  When the   zero multiplicity of each $g_i$ is at least $\ell$,  then $T_{\frac{g_i'}{g_i} }(r)\le \frac 2\ell T_{ \mathbf{g}}(r) $ and hence the coefficients of $D_{\mathbf{g}}(G)$ are in the same growth w.r.t. $\mathbf{g}$.  Therefore, we will need to extend the  GCD theorem  in \cite{levin2019greatest} to the case where the coefficients of the polynomial are in the same growth as the entire curves.   In this step,  we can conclude (i) under the assumption that $T_{(\frac{g_1}{g_0})^{m_1}(\frac{g_2}{g_0})^{m_2}}\le \frac c\ell T_{ {\mathbf{g}}}(r) $ for some computable constants $m_1$, $m_2$ and $c$ if the   zero multiplicity of each $g_i$ is sufficiently large. The next step is to do an algebraic reduction, which also enables us  to find the exceptional sets  for Theorem \ref{main_thm_1} explicitly.

	Some background materials will be given in the next session.  In Section \ref{largemultiplicity}, we develop some lemmas to deal with orbifold curves with sufficiently large multiplicities.  In Session \ref{main_lem}, we formulate a version of Nevanlinna's second main theorem with moving targets of the same growth and establish the corresponding GCD theorem.
	The proof of Theorem \ref{main_thm_1} will be given in Section \ref{thm1}, and 
	the proofs of the other theorems will be given in Section  \ref{others}.  Finally, we explain in Section \ref{remarks} how to adapt our proofs of 
	Theorem \ref{GG_conj} and  Theorem \ref{finitemorphism} to show the strong  Green-Griffiths-Lang conjecture, i.e. 
	finding   exceptional sets under the assumption that the multiplicities $m_i=\infty$, $1\le i\le 3$ in both theorems.

	\section{Preliminaries}\label{preliminaries}
	We will give relevant materials and derive some basic results in this session.
	
	\subsection{Nevanlinna Theory}
	We will set up some notation and definitions  in Nevanlinna theory and recall some basic results.   We refer to \cite{vojta2009diophantine}, \cite{ru2021nevanlinna}, and \cite{GSW20} for details.

	Let $f$ be a meromorphic function  and   $z\in \CC$ be a complex number. Denote $v_z(f):=\ord_z(f)$,
	$$v_z^+ (f):=\max\{0,v_z(f)\}, \quad\text{and }\quad  v_z^- (f):=-\min\{0,v_z(f)\}.$$
	\begin{align*}
		N_f(\infty,r)
		=\sum_{0<|z|\le r } v_z^- (f)\log |\frac{r}{z}|+v_0^- (f)\log r,
	\end{align*}
	and
	\begin{align*}
		N^{(Q)}_f(\infty,r)
		=\sum_{0<|z|\le r } \min\{Q,v_z^- (f)\}\log |\frac{r}{z}|+\min\{Q,v_0^- (f)\}\log r.
	\end{align*} 
	Then define the {\it counting function} $N_f(r,a)$ and the {\it truncated counting function} $N^{(Q)}_f(r,a)$ for $a\in\CC$ as
	$$
	N_f(a,r):=N_{1/(f-a)}(r, \infty)\quad\text{and}\quad N^{(Q)}_f(a,r):=N^{(Q)}_{1/(f-a)}(\infty,r).
	$$
	The  {\it proximity function} $m_f(\infty,r)$ is defined by
	$$
	m_f(\infty,r):=\int_0^{2\pi}\log^+|f(re^{i\theta})|\frac{d\theta}{2\pi},
	$$
	where $\log^+x=\max\{0,\log x\}$ for  $x\ge 0$. For any $a\in \CC,$ the {\it proximity function} $m_f(a,r)$ is defined by
	$$m_f(a,r):=
	m_{1/(f-a)}(\infty,r).
	$$
	The {\it characteristic function} is defined by
	$$
	T_f(r):=m_f(\infty,r)+N_f(\infty,r).
	$$

Let ${\mathbf f}: \CC \rightarrow \PP^n(\CC)$ be a  holomorphic map and $ (f_0,\dots,f_n)$ be a reduced representation of  ${\mathbf f}$, i.e. $f_0,\dots, f_n$ are  entire functions on $\CC$ without common zeros. The Nevanlinna-Cartan {\it characteristic function}  $T_{\mathbf f}(r)$ is defined by
$$T_{\mathbf f}(r) =   \int_0^{2\pi} \log\max\{|f_0(re^{i\theta})|,\dots ,|f_n(re^{i\theta})|\} \frac{d\theta}{2\pi}.$$
This definition is independent, up to an additive constant, of the choice of the reduced representation of ${\mathbf f}$. 

We will make use of the following   elementary inequality . 
\begin{proposition} 
	\label{basic_prop}
	Let ${\mathbf f}=[f_0:\dots:f_n]:\mathbb C\to \PP^n(\mathbb C)$ be  holomorphic curve, where $f_0,\dots,f_n$ are entire functions without common zeros.  Then  
	\begin{equation*}
		T_{f_j/f_i}(r)+O(1)\leq T_{\mathbf f}(r)\leq \sum_{j=0}^nT_{f_j/f_0}(r)+O(1).
	\end{equation*}
\end{proposition}

Recall the following   truncated second main theorem due to Ru and the second author.

\begin{theorem}[{\cite[Theorem 2.1]{ru2004truncated}}]\label{trunborel}  
	Let $\mathbf{f}=(f_0,\hdots,f_n):\CC\to\PP^n(\CC)$ be a holomorphic map with $f_0,\dots,f_n$ entire and no common zeros. Assume that $f_{n+1}$ is a holomorphic function satisfying the equation $f_0+\dots+f_n+f_{n+1}=0$. If $\sum_{i\in I}f_i\ne 0$ for any proper subset $I\subset\{0,\dots,n+1\}$, then 
	\begin{equation*}
		T_{\mathbf{f}}(r)\leq_{\exc} \sum_{i=0}^{n+1} N_{f_i}^{(n)}(0,r)+O(\log T_{\mathbf{f}}(r)).
	\end{equation*}
\end{theorem}
We will use the following second main theorem for hypersurfaces with  truncation and bounded degeneration degree, which can be obtained from  \cite{TruSMTAn} easily.  

\begin{theorem}[\cite{TruSMTAn}]\label{SMTmoving}
	Let $\mathbf{f}$ be a nonconstant holomorphic map of $\CC$ into $\PP^n$.     Let $ D_i$, $1\le i\le q$, be  hypersurfaces in $\PP^n(\mathbb C)$  of degree $d_i$, in general position.   Let $0<\epsilon<1$.  Then there exist two positive integers $M(\ge O(\epsilon^{-2}))$ and $N(=O(\epsilon^{-1}))$ depend only on $\epsilon$, $n$ and  $d_i$, $1\le i\le q$,
	such that  for any  holomorphic map of $\mathbf{f}:\CC\to\PP^n$ either the following inequality holds:
	\begin{equation*} 
		(q-n-1-\epsilon)T_{\mathbf{f}}(r)\le_{\exc} \sum_{i=1}^q\frac{1}{d_i} N_{\mathbf{f}}^{(M)}(D_i,r),
	\end{equation*}
	or the image of $\mathbf{f}$ is contained in a hypersurface in $\PP^n(\mathbb C)$	with degree bounded by $N$.	
\end{theorem}


\section{Orbifold curves with sufficiently large  multiplicities}\label{largemultiplicity}

We show some basic propositions and develop a version of Borel lemma for orbifold curves with sufficiently large  multiplicities.

\begin{proposition}
	\label{coutingzero}  Let $f$ be a non-constant entire function on $\CC$. 
	Suppose that the zero  multiplicity of  $f$   at  each $z\in\mathbb C$  is either zero or bigger than $\ell\ge 1$. Then		\begin{equation*}
		T_{f'/f}(r)\le_{\exc}\frac   1\ell T_f(r)+ O(\log T_f(r)).
	\end{equation*}
\end{proposition}
\begin{proof}
	The assertion follows from the lemma of logarithmic and the following estimate:
	\begin{align*}
		N_{f'/f}(\infty, r)=N_f^{(1)}(0,r) \le   \frac 1\ell N_f(0,r) \le \frac   1\ell T_f(r)+O(1).
	\end{align*}
\end{proof}
\begin{proposition}\label{polyheight}
	Let $f_0,\hdots,f_n$ be non-constant entire functions with no common zeros and the zero multiplicity of each $f_i$ be either zero or bigger than a positive integer $\ell$.   
	Let  $u_i=f_i/f_0$ for $1\le i\le n$. Then for any 
	$\alpha\in \mathbb C(\frac{u_1'}{u_1},\hdots,\frac{u_n'}{u_n})$, there is a positive  constant $c$ independent of $\ell$ such that 
	$T_{\alpha}(r)\le_{\exc}   \frac{c} \ell \cdot T_{\mathbf{f}}(r),$
	where  $\mathbf{f}:=(f_0,\hdots,f_{n})$. 
\end{proposition}
\begin{proof}
	Let $\tilde{\mathbf{u}}:=[1:\frac{u_1'}{u_1}:\dots:\frac{u_n'}{u_n}]$.
	Since $\alpha\in \mathbb C(\frac{u_1'}{u_1},\hdots,\frac{u_n'}{u_n})$, we may find two  coprime homogeneous polynomials $P,Q\in \mathbb C[x_0,\hdots,x_n]$ such that $\alpha=P(\tilde{\mathbf{u}})/Q(\tilde{\mathbf{u}})$.
	Then 
	\begin{align}\label{alphabound}
		T_{\alpha}(r)&\le T_{P(\tilde{\mathbf{u}})}(r)	+T_{Q(\tilde{\mathbf{u}})}(r)\le (\deg P+\deg Q)T_{\tilde{\mathbf{u}}}(r).
	\end{align}				
	By Proposition \ref{basic_prop}, we have			
	\begin{equation*}
		T_{\tilde{\mathbf{u}}}(r) \le \sum_{j=1}^n T_{\frac{u_j'}{u_j}}(r)+O(1). 
	\end{equation*}	
	By the lemma of logarithmic, we have 
	\begin{equation*}
		m_{\frac{u_j'}{u_j}}(\infty,r)\le_{\exc} O(\log T_{u_j}(r)) \le O(\log T_{\mathbf f} (r)).
	\end{equation*}
	For the properties of counting functions, we have		
	\begin{align*}
		N_{\frac{u_j'}{u_j}}(\infty, r)&\le N_{u_j}^{(1)}(0,r)+N_{u_j}^{(1)}(\infty,r)\le N_{f_j}^{(1)}(0,r)+N_{f_0}^{(1)}(0,r) \\
		&\le\frac 1\ell N_{f_j}(0,r) + \frac 1\ell N_{f_j}(0,r) \le \frac  2\ell T_{\mathbf f}(r).
	\end{align*}
	Therefore, 
	$T_{\tilde{\mathbf{u}}}(r) \le_{\exc}  \frac  {2n}\ell T_{\mathbf f}(r)+O(\log T_{\mathbf f} (r))$,	
	and
	\begin{align}
		T_{\alpha}(r)&\le_{\exc} \frac {2n (\deg P+\deg Q)}\ell T_{\mathbf{f}}(r)+O(\log T_{\mathbf f} (r)).
	\end{align}	 
\end{proof}

We also need the following   version  of the Borel Lemma for orbifold curves.   
\begin{lemma}
	\label{Mborel1}
	Let $f_0,\hdots,f_n$ be  non-constant entire functions with no common zeros, and $(a_0,\hdots,a_n)\ne(0,\hdots,0)$ be an $(n+1)$ tuple of  meromorphic functions.  Assume that   the zero multiplicity of each $f_i$ is either zero or bigger than a positive integer $\ell$.
	Suppose that $a_0f_0+a_1f_1+\hdots+a_nf_n=0$. Then for each $i$  with $a_i\ne 0$, there exists $j\ne i$  such that 
	$$
	T_{f_i/f_j}(r)\leq_{\exc} 3n\cdot  T_{\mathbf a}(r)+\frac{n^2-1}{\ell}T_{\mathbf{f}}(r)+O(\log T_{\mathbf{f}}(r)), 
	$$
	where ${\mathbf a}=[a_{0}:\cdots: a_{n}]$, and  $\mathbf{f}:=(f_0,\hdots,f_{n})$. 
\end{lemma}
\begin{proof} 
	By multiplying an appropriate meromorphic function to each $a_i$, we may assume that  the non-trivial $a_i$, $0\le i\le n$, are entire functions without common zeros.  Let $H_i$ be the coordinate hyperplane defined by $x_i=0$, $0\le i\le n$.  Then
	\begin{align}\label{zeroai}	N_{a_i}(0,r)=N_{\mathbf a}(H_i,r)\le T_{\mathbf a}(r)+O(1).
	\end{align}
	For a given $i$ with non-trivial $a_i$, there exists a vanishing subsum  of $a_0f_0+\hdots+a_nf_n=0$ consisting of the term $a_if_i$ and without any   vanishing proper subsum.  By reindexing, we may assume that $i=1$ and this vanishing subsum is
	\begin{equation}\label{propersubsum2}
		a_0f_0+\cdots+a_m f_m=0.
	\end{equation}
	If $m=1$, then $T_{f_1/f_0}(r)=T_{a_0/a_1}(r)\le T_{\mathbf a}(r)$.  Therefore we assume that $m\ge 2$. 
	Let $g$   be an entire function such that $\tilde f_0:=f_0 /g, ,\cdots,\tilde f_{m-1}:=f_{m-1} /g$ are entire functions with no common zeros. 
	Let 
	$$
	\tilde{\mathbf f}:=(\tilde f_0,\hdots,\tilde f_{m-1}).
	$$
	Let $h$ be an entire function such that $a_0\tilde f_{0} /h, \cdots,a_{m-1}\tilde f_{m-1} /h$ are entire functions with no common zeros.   Let 
	$$
	\mathbf{F}:=(\frac{a_0 \tilde f_0}h,\hdots,\frac{a_{m-1} \tilde f_{m-1}}h).
	$$ 
	We can deduce from the definition of characteristic functions that	
	\begin{equation}
		T_{\mathbf{F}}(r)\leq T_{\tilde{\mathbf f}}(r)+T_{\mathbf a}(r)\leq T_{\mathbf f}(r)+T_{\mathbf a}(r).
	\end{equation}
	On the other hand, we may write $\tilde f_i =\frac{h}{a_i}\cdot\frac{a_i \tilde f_i}h$.  
	Then
	\begin{equation*}
		\max_{0\leq i\leq m-1}\{\log | \tilde f_i|\}\leq \max_{0\leq i\leq m-1}\{\log |a_i \tilde f_i/h|\}+\log|h|+\max_{0\leq i\leq m-1}\{\log 1/|a_i|\}.
	\end{equation*}
	Hence,
	\begin{equation}\label{Charf_k2}
		T_{\tilde{\mathbf f}}(r)\leq T_{\mathbf{F}}(r)+(m-1) T_{\mathbf a}(r)+N_h(0,r).
	\end{equation}
	Since 	$\tilde f_0,\hdots,\tilde f_{m-1}$ have no common zeros, the zeros of $h$ must be zeros of some $a_i$, $0\le i\le m-1$.  Therefore,	
	$$
	N_h(0,r)\le \sum_{i=0}^{m-1} N_{a_i}(0,r)\le m T_{\mathbf a}(r)+O(1)
	$$		 
	by 	\eqref{zeroai}.	 Consequently,  we derive from \eqref{Charf_k2} that 
	\begin{equation}\label{Charf_k3}
		T_{\tilde{\mathbf f}}(r)\leq T_{\mathbf{F}}(r)+(2m-1) T_{\mathbf a}(r)+O(1) .
	\end{equation}

	Applying Theorem \ref{trunborel} to the map $\mathbf{F}$ with the equation
	\begin{equation}\label{propersubsum3}
		\frac{a_0 \tilde f_0}h  +\cdots+\frac{a_m \tilde f_m}h  =0,
	\end{equation} we have
	\begin{equation*}\label{truncation2}
		\begin{aligned}
			T_{\mathbf{F}}(r)&\leq_{\exc}  \sum_{i=0}^{m} N_{a_i  \tilde f_i/h}^{(m-1)}(0,r) +O(  \log T_{\mathbf{F}}(r))\\
			&\le_{\exc} \sum_{i=0}^{m} N_{a_i}(0,r) +\sum_{i=0}^{m} N_{ f_i }^{(m-1)}(0,r)+O( \log T_{\mathbf{f}}(r))\qquad(\text{as } \tilde f_i=f_i/g )\\
			&\le_{\exc}  (m+1) T_{\mathbf a}(r) +\frac{m-1}{\ell}\sum_{i=0}^{m} N_{ f_i }  (0,r)+O( \log  T_{\mathbf{f}}(r))\qquad(\text{by \eqref{zeroai}})\\
			&\le_{\exc}  (m+1) T_{\mathbf a}(r) +\frac{m^2-1}{\ell}T_{\mathbf{f}}(r)+O( \log  T_{\mathbf{f}}(r)).
		\end{aligned}
	\end{equation*}
	Together with Proposition \ref{basic_prop}, \eqref{Charf_k3} and  the fact $m\le n$, this yields
	\begin{equation}
		T_{f_1/f_j}(r)\le T_{\tilde{\mathbf f}}(r)+O(1) \leq_{\exc} 3n\cdot T_{\mathbf a}(r)+\frac{n^2-1}{\ell}T_{\mathbf{f}}(r)+O(\log  T_{\mathbf{f}}(r))
	\end{equation}
	for any $0\le j\le m$.
\end{proof}
Let $Q \in\mathcal M[x_1,\hdots,x_n]$, where $\mathcal M$ is the field of meromorphic functions. We may express  
$Q =\sum_{{\mathbf i}\in I_Q } a_{\mathbf i}\cdot{\mathbf x}^{\mathbf i}$, where ${\mathbf i}=(i_1,\hdots,i_n)$, ${\mathbf x}^{\mathbf i}=x_1^{i_1}\hdots x_n^{i_n}$, and $a_{\mathbf i}\ne 0$ if ${\mathbf i}\in I_Q $.  
We let 
\begin{align}\label{norm}
	\| Q \|_z:=\max_{{\mathbf i}\in I_A}\{|a_{\mathbf i}(z)|\},
\end{align}
and define  the characteristic function of $Q$ as 
\begin{align}\label{charF}
	T_{Q }(r):=T_{[\cdots:a_{\mathbf i}:\cdots]}(r),
\end{align}
where $a_{\mathbf i}$ is taken for every ${\mathbf i}\in I_Q $.

\begin{corollary}\label{MborelCor}
	Let $f_0,\hdots,f_n$ be non-constant entire functions with no common zeros and the zero multiplicity of each $f_i$ be either zero or larger than a positive integer $\ell$.   
	Let  $Q$  be a non-constant homogeneous polynomial in $\mathbb C(\frac{u_1'}{u_1},\hdots,\frac{u_n'}{u_n})[x_0,\hdots,x_n]$, where $u_i=f_i/f_0$ for $1\le i\le n$.  Then 
	\begin{align}\label{heightA}
		T_{Q }(r)\le_{\exc}   \frac{c_1} \ell \cdot T_{\mathbf{f}}(r)\quad\text{ for some positive  constant $c_1$ independent of $\ell$},
	\end{align}
	where $\mathbf{f}=(f_0,\hdots,f_n)$.
	Furthermore, if  $Q (f_0,\hdots,f_n)=0$, then there exists a non-trivial $n $-tuple of integers $(j_1,\hdots,j_n)$ with $|j_1|+\cdots+|j_n|\le 2\deg Q $ and  a positive real $c_2$  independent of $\ell$ 
	such that  
	$$
	T_{ u_1^{j_1}\cdots  u_n ^{j_n} }(r)\le_{\exc} \frac {c_2}{\ell} \cdot T_{\mathbf{f}}(r).		
	$$ 
\end{corollary}
\begin{proof}
	Let $m=\deg Q $ and $Q =\sum_{{\mathbf i}\in I_Q}  a_{\mathbf i}\cdot{\mathbf x}^{\mathbf i}$, where ${\mathbf i}=(i_0,\hdots,i_n)$, with $|{\mathbf i}|=i_0+\cdots+i_n=m$,  ${\mathbf x}^{\mathbf i}=x_0^{i_0}\cdots x_n^{i_n}$,  and $a_{\mathbf i}\ne 0,$ if ${\mathbf i}\in I_Q $.  Let ${\mathbf i}_0\in I_Q $.
	By Proposition \ref{basic_prop}, we have
	$$
	T_Q (r)\le \sum_{{\mathbf i}\in I_Q} T_{\frac{a_{\mathbf i}}{a_{{\mathbf i}_0}}}(r).
	$$
	Then the first assertion follows from Corollary \ref{polyheight}.
	
	If $Q (f_0,\hdots,f_n)=0$, then 
	$ \sum_{{\mathbf i}\in I_Q } a_{\mathbf i} \cdot f_0^{i_0}\cdots f_n^{i_n}=0$.  
	Let ${\mathbf F}=[\cdots,f_0^{i_0}\cdots f_n^{i_n},\cdots]$, where the index set ${\mathbf i}$ runs over all $ {i_0}+\cdots +{i_n}=m$.   Then the zero multiplicities of $f_0^{i_0}\cdots f_n^{i_n}$ are at least $\ell$ if they are not zero.   We note that the entries of ${\mathbf F}$ have no common zeros and  $T_{\mathbf F}(r)\le m T_{\mathbf{f}}(r)+O(1)$.
	Then by   Lemma \ref{Mborel1} and \eqref{heightA}, we find ${\mathbf i}\in I_Q$ distinct from ${\mathbf i}_0=(i_0',\hdots,i_n')$ such that
	\begin{align}\label{Tui}
		T_{f_0^{i_0-i_0'}\cdots f_n^{i_n-i_n'}}(r)\le_{\exc}   \frac{c_2}{\ell} \cdot T_{\mathbf{f}}(r),
	\end{align}
	for some positive constant $c_2$ independent of $\ell$.
	Finally, since $|{\mathbf i}|=|{\mathbf i}_0|$  and $u_i=f_i/f_0$,  we have 
	$f_0^{i_0-i_0'}\cdots f_n^{i_n-i_n'}=u_1^{i_1-i_1'}\cdots u_n^{i_n-i_n'}$ to
	conclude the second assertion by \eqref{Tui}. 
\end{proof}

\section{GCD theorem}\label{main_lem}	

\subsection{GCD with moving targets of the same growth}\label{gcd}
Let $f$ and $g$ be meromorphic functions.  We let
\begin{equation*}
	n(f,g,r):=\sum_{|z|\le r}\min\{v_z^+(f),v_z^+(g)\}
\end{equation*}
and 
\begin{equation*}
	N_{\gcd}(f,g,r):=\int_0^r \frac{n(f,g,t)-n(f,g,0)}{t}dt+ n(f,g,0)\log r.
\end{equation*}
We will need the following  GCD theorem. 

\begin{theorem} \label{movinggcd}
Let $n\ge 2$ be a positive integer. Let  $g_0,\dots,g_n$ be non-constant entire functions with no common zeros and the zero multiplicity of each $g_i$ be either zero or larger than a positive integer $\ell$. Let $u_i=g_i/g_0$ for $1\le i\le n$.  Let $F$ and $G$ be coprime homogeneous polynomials in $\mathbb C(\frac{u_1'}{u_1},\hdots,\frac{u_n'}{u_n})[x_0,\hdots,x_n]$  such that not both of them are identically zero at $(1,0,\hdots,0),\hdots,(0,\hdots,0,1)$. Then for any  sufficiently small  $\epsilon>0$,
there exist  integers $N\ge O(\epsilon^{-3})$ and $m=O(\epsilon^{-1})$, both depending only on $\epsilon$, such that if $\ell>N$, then  we have  either
	\begin{align}\label{gcd0}
		N_{\rm gcd} (F(g_0,\hdots,g_n),G(g_0,\hdots,g_n),r)\le_{\rm exc}  
		\epsilon T_{\bf g}(r),  
	\end{align} 
	or 
	\begin{align}\label{gdegenerate}
		T_{(\frac{g_1}{g_0})^{m_1}\cdots (\frac{g_n}{g_0})^{m_n}}\le_{\rm exc}   \epsilon^3    T_{\bf g}(r)
	\end{align}
	for some non-trivial tuple of integers $(m_1,\hdots,m_n)$ with $|m_1|+\cdots+|m_n|\le 2m$, where ${\bf g}=(g_0,\hdots,g_n)$.  
\end{theorem}
Let $K_{{\bf g}}:=\{a: a \text{ is a meromorphic function with }T_{a}(r)\le {\rm o} (T_{\mathbf g}(r))\}$.
In \cite{levin2019greatest}, Levin and the second author showed that \eqref{gcd0} holds for $F$ and $G$ with coefficients in $\mathbb C$ or $K_{{\bf g}}$ under the assumption that $g_i$, $0\le i\le n$, are entire functions with no zeros and they are multiplicatively independent over $\mathbb C$ or $K_{{\bf g}}$ respectively.  In view of Proposition \ref{polyheight}, the coefficients of $F$ and $G$ is in $ \mathbb K_{\bf g}:=\{a: a \text{ is a meromorphic function with }T_{a}(r)\le {\rm O} (T_{\mathbf g}(r))\},$ which is a field of the same growth w.r.t. ${\bf g}$.   
Therefore,  we need to develop a Nevanlinna's second main theorem with moving targets of the same growth in order to adapt the proof in \cite{levin2019greatest}.

\subsection{Nevanlinna Theory with Moving Targets of the Same Growth}
In view of the second main theorem for function fields with moving targets developed in \cite{Wirsing} and \cite{Wa2004}, it comes naturally to  develop a Nevanlinna's second main theorem with moving targets of the same growth.  We will follow the ideas in \cite{Wirsing}.
Let ${\mathbf f}=(f_0,\dots,f_n)$ be a holomorphic map from $\mathbb C$ to  $\PP^n$ where $f_0,f_1,\hdots,f_n$ are  entire functions without common zeros.  
Let $a_{0}, \cdots,a_{n}$ be meromorphic functions  in $\mathbb K_{\bf f}$, and $L:=a_{0} X_0+\dots+a_{n} X_n$.  
Then  $L$ defines a  hyperplane $H$ in $\PP^n $, over the field $\mathbb K_{\bf f}$.  We note that  $H(z)$ is the hyperplane determined by the linear form $L(z) =a_{0}(z) X_0+\dots+a_{n}(z) X_n$   for $z\in\mathbb C$ that is not a common zero of $a_{0}, \cdots,a_{n}$, or a pole of any $a_{k}$, $0\le k\le n$.
The definition of the Weil function, proximity function and counting function can be easily extended to moving hyperplanes.  
For example,
\begin{align*}
	\lambda_{H(z)}(P)=-\log \frac{|(ha_{0})(z)x_0+\cdots+(ha_{n})(z)x_n|}{\max\{|x_0|,\dots ,|x_n|\}\max\{|(ha_{0})(z)|,\dots ,|(ha_{n})(z)|\}},
\end{align*}
where $h$ is a meromorphic function such that  $ha_{0}, \cdots,ha_{n} $ are entire functions without common zeros, $P=(x_0,\hdots,x_n)\in \PP^n(\mathbb C)$ and $z\in\mathbb C$.
It's clear that 
\begin{align}\label{Weil}
	\lambda_{H(z)}(P)=-\log \frac{| a_{0} (z)x_0+\cdots+ a_{n} (z)x_n|}{\max\{|x_0|,\dots ,|x_n|\}\max\{| a_{0} (z)|,\dots ,| a_{n} (z)|\}},
\end{align}
for $z\in\mathbb C$ which is not a common zero of $a_{0}, \cdots,a_{n}$, or a pole of any $a_{k}$, $0\le k\le n$.
The first main theorem for a moving hyperplane $H$  can be stated as
\begin{equation}\label{fmtmov}
	T_{\mathbf{f}}(r) =N_{\mathbf{f}}(H,r)+m_{\mathbf{f}}(H,r)+ T_{\mathbf a}(r)+O(1), 
\end{equation}
where ${\mathbf a}:=[a_0:\cdots:a_n]$. 

We will reformulate the second main theorem with moving targets stated in \cite[Theorem A6.2.1]{{ru2021nevanlinna}} to the situation where the coefficients of the underlying linear forms are in $\mathbb K_{\bf f}$. 
Let  $a_{j0}, \cdots,a_{jn}\in  \mathbb K_{\bf f}$ and let $L_j:=a_{j0} X_0+\dots+a_{jn} X_n$.	
Without loss of generality, we will normalize the linear forms $L_j$, $1\le j\le q$,  such that for each $1\le j\le q$, there exists $0\le j'\le n$ such that $a_{ij'}=1$.  Let $t$ be a positive integer and let $V(t)$ be the complex vector space spanned by the elements
\begin{align*}
	\left\{ \prod a_{jk}^{n_{jk}} : n_{jk}\ge 0,\,\sum n_{jk}=t \right\},
\end{align*}
where the products and sums run over $1\le j\le q$ and $0\le k\le n$.
Let $1=b_1,\cdots,b_u$ be a basis of $V(t)$ and $b_1,\cdots,b_w$ a basis of $V(t+1)$.  It's clear that $u\le w$.
Moreover, we have 
\begin{align}\label{vlimit}
	\liminf_{t\to\infty}\dim V(t+1)/\dim V(t)=1. 
\end{align}
\begin{definition}
	Let $E $ be a $\mathbb C$-vector space spanned by finitely many meromorphic functions.  We say that an analytic map ${\mathbf f}=(f_0,\dots,f_n):\mathbb C\to\mathbb P^n$ is linearly nondegenerate over $E$ if whenever we have a linear combination 
	$\sum_{i=1}^{m}a_{i}f_{i}=0$ with $a_{i}\in E$, then $a_{i}=0$ for each $i$;
	otherwise we say
	that ${\mathbf f}$ is  linearly degenerate over $E$. 
\end{definition}

The following formulation of the second main theorem with moving targets in  $\mathbb K_{\bf f}$ follows from the proof of  \cite[Theorem A6.2.1]{{ru2021nevanlinna}} by adding the Wronskian term and computing the error term explicitly when applying the second main theorem.	 
\begin{theorem}
	\label{movingsmt0}
	Let ${\mathbf f}=(f_0,\hdots,f_n):\mathbb{C}\to\mathbb{P}^n(\mathbb{C})$ be a holomorphic curve  where $f_0,f_1,\hdots,f_n$ are  entire functions without common zeros.  Let $H_j$, $1\le j\le q$, be arbitrary (moving) hyperplanes given by $L_j:=a_{j0} X_0+\dots+a_{jn} X_n$ where $a_{j0}, \cdots,a_{jn}\in  \mathbb K_{\bf f}$.  Denote by $W$ the Wronskian of $\{hb_mf_k\,|\, 1\le m\le w,\,  0\le k\le n\}$, where $h$ is a meromorphic function such that $hb_1,\hdots, hb_w$ are entire functions without common zeros.  If ${\mathbf f}$ is linearly non-degenerate over $ V(t+1) $, then for any $\varepsilon>0$, we have the following inequality:
	\begin{align*} \int_0^{2\pi} \max_J \sum_{k\in J}&\lambda_{H_k(re^{i\theta})}({\mathbf f}(re^{i\theta}))\frac{d\theta}{2\pi}+\frac 1uN_{W}(0,r)\\
		&\leq_{\operatorname{exc}}   \left(\frac wu (n+1)+\epsilon\right)T_{\mathbf{f}}(r)+\frac wu (t+2) \max_{1\le j\le q } T_{\mathbf a_{j}}(r),
	\end{align*}  where $w=\dim V(t+1)$, $u=\dim V(t)$, ${\mathbf a_j}=[a_{j0}:\cdots: a_{jn}]$ and the maximum is taken over all subsets $J$ of $\{1,\dots, q\}$ such that $H_j(re^{i\theta})$, $j\in J$, are in general position. 
\end{theorem}

\subsection{The key theorem and the proof of Theorem \ref{movinggcd}	}
The following is the key theorem to prove Theorem \ref{movinggcd}.  
\begin{theorem}\label{Mfundamental}
	Let $n\ge 2$ be a positive integer. Let  $g_0,\hdots,g_n$ be entire functions without common zeros and let ${\bf g}=(g_0,\hdots,g_n)$.
	Let $F,G $ be coprime homogeneous polynomials in $n+1$ variables of the same degree $d>0$  over $\mathbb K_{\bf g}$.  Assume that one of the coefficients in each expansion of $F$ and $G$ is 1.
	Let $I$ be the set of exponents ${\bf i}$ such that ${\bf x}^{\bf i}$ appears with a nonzero coefficient in either $F$ or $G$.  
	For every positive integer $t$,
	we denote by $V_{F,G}(t)$ the (finite-dimensional) $\mathbb C$-vector
	space  spanned by $\prod_{\alpha}\alpha^{n_{\alpha}}$,
	where $\alpha$ runs through all non-zero coefficients of $F$ and $G$, $n_{\alpha}\ge0$ and $\sum n_{\alpha}=t$; we also put
	$d_{t}:=\dim V_{F,G}(t)$.  
	For every  integer  $m\ge d$, we let $M=M_{m,n,d}=2\binom {m+n-d}{n}- \binom {m+n-2d}{n}$.    
	Then there exists a positive real  $c$ and a positive integer  $b$  such that if the set $\{g_0^{i_0}\cdots g_n^{i_n}: i_0+\cdots+i_n=m\}$ is linearly non-degenerate over $V_{F,G}(Mb+1)$, then   we have that  the following estimate.
	\begin{align*}
		MN_{\rm gcd}&(F({\bf g}),G({\bf g}),r) 
		\le_{\rm exc}  
		\left(M'+\frac{d_{Mb}}{d_{M(b-1)}}M-M\right)mnT_{\bf g}(r)\\
		&+c_{m,n,d}  \sum_{i=0}^n N^{(L)}_{ g_i}(0,r)+ \left(\frac{m}{n+1}\binom{m+n}{n}-c_{m,n,d}-M'm\right)\sum_{i=0}^nN_{g_i}(0,r)\\
		& + \binom {m+n-2d}{n}N_{\rm gcd}(\{{\bf g}^{\bf i}\}_{{\bf i}\in I},r) 
		+  c(T_{ F}(r)+T_{  G}(r)) +O(\log T_{\mathbf{g}}(r)),
	\end{align*}
	where $c_{m,n,d}=2\binom {m+n-d}{n+1}- \binom {m+n-2d}{n+1}$,
	$M'$ is an integer of order $O(m^{n-2})$,  $L=\frac12M(M-1)c_{m,n,d}^{-1}$  and $c  =\frac{d_{Mb}}{d_{M(b-1)}}(1+M(b+1))$.
\end{theorem}

\begin{proof}
	Most of the proof is identical to the  one of \cite[Theorem 5.7]{levin2019greatest},  except that it is necessary to estimate the characteristic functions of the linear forms when apply  Theorem \ref{movingsmt0}.  We will point out the important differences and omit the proof which is similar to the one of \cite[Theorem 5.7]{levin2019greatest}.   We also refer to \cite[Theorem 26]{GSW} for the structure of the proof and the explicit computations of all constants.

	Since one of the coefficients in the expansion of $F$ and $G$ is $1$, from the definition \eqref{norm} we have
	\begin{equation}\label{log_FG}
		\log\|F\|_z\ge 0 \text{\qquad and\qquad } \log\|G\|_z\ge 0
	\end{equation}
	for each $z\in\CC$ which is neither a zero nor a pole of the coefficients of $F$ and $G$. 
	Let $(F,G)$ be  the ideal generated by $F$ and $G$ in $\mathbb K_{\bf g}[\mathbf{x}]$.  If $(F,G)=1$, then it is elementary to show that $N_{\rm gcd} (F({\bf g}),G({\bf g}),r)\le c (T_{F}(r)+T_G(r)) $, where $c$ is a positive constant independent of ${\bf g}$.  Therefore, we assume that ideal $(F,G)$ is proper in  $\mathbb K_{\bf g}[\mathbf{x}]$. 	Let $(F,G)_m:=\mathbb K_{\bf g}[\mathbf{x}]_m\cap (F,G)$, where 
	$$\mathbb K_{\bf g}[\mathbf{x}]_m:=\{P\in \mathbb K_{\bf g}[\mathbf{x}]: P\text{ is a homogeneous polynomial of degree } m\}.$$ 	We choose $\{\phi_1,\dots,\phi_M\}$ to be a basis of the $\mathbb K_{\bf g}$-vector space $(F,G)_m$ consisting of elements of the form $F\mathbf{x^i},G\mathbf{x^i}$. 
	For each $z\in\CC$, we can construct a basis $B_z$ of $V_m:=\mathbb K_{\bf g}[\mathbf{x}]_m/(F,G)_m$ with monomial representatives $\mathbf{x}^{\mathbf{i}_1},\dots,\mathbf{x}^{\mathbf{i}_{M'}}$ as in the proof of \cite[Theorem 5.7]{levin2019greatest}. Let $I_z:=\{\mathbf{i}_1,\dots,\mathbf{i}_{M'}\}$. 
	For each $\mathbf{i}\notin I_z$, $|\mathbf{i}|=m$, we have 
	\begin{equation}\label{initial}
		{\mathbf x}^{\mathbf i}+\sum_{\mathbf j\in I_z}c_{z,{\mathbf j}}{\mathbf x}^{{\mathbf j}}\in(F,G)_{m}
	\end{equation}
	for some choice of $c_{z,{\mathbf i},j}\in \mathbb K_{\bf g}$.
	Then for each such $\mathbf{i}$, there is a linear form
	\begin{equation}\label{linear_form}
		L_{z,\mathbf{i}}:=\sum_{\ell=1}^M b_{z,{\mathbf i},\ell}y_{\ell}\in \mathbb K_{\bf g}[y_1,\dots,y_M]
	\end{equation}
	such that 
	\begin{equation}\label{linear_form_phi}
		L_{z,\mathbf{i}}(\phi_1,\dots,\phi_M)=c_z\left(\mathbf{x^i}+\sum_{\mathbf j\in I_z}c_{z,{\mathbf j}}{\mathbf x}^{{\mathbf j}}\right),
	\end{equation}
	where $c_z\in \mathbb K_{\bf g}^*$ will be chosen later. 
	By the choice of $\phi_\ell$, we may write 
	\begin{equation}\label{phi}
		\phi_\ell=\sum_{\mathbf{j}\notin I_z, |\mathbf{j}|=m} \alpha_{z,\ell,\mathbf{j}} \mathbf{x}^{\mathbf{j}}+\sum_{\mathbf{j}\in I_z}\alpha_{z,\ell,\mathbf{j}}\mathbf{x^j}, 
	\end{equation}
	where each $\alpha_{z,\ell,\mathbf{j}}$, $|\mathbf{j}|=m$, is a coefficient of either $F$ or $G$, thus 
	\begin{equation*}
		\max_{|\mathbf{j}|=m}\{\log|\alpha_{z,\ell,\mathbf{j}}|\}\le  \log\|F\|_z+\log\|G\|_z
	\end{equation*}
	for each $\ell$ and $\mathbf{j}$. Let $\Phi:=(\phi_1,\dots,\phi_M)$. 
	Combining \eqref{linear_form} and \eqref{phi}, we have 
	\begin{equation}\label{linear_form_phi_2}
		L_{z,\mathbf{i}}(\Phi(\mathbf{x}))=\sum_{\ell=1}^M b_{z,{\mathbf i},\ell}\left( \sum_{\mathbf{j}\notin I_z, |\mathbf{j}|=m} \alpha_{z,\ell,\mathbf{j}} \mathbf{x}^{\mathbf{j}}+\sum_{\mathbf{j}\in I_z}\alpha_{z,\ell,\mathbf{j}}\mathbf{x^j} \right).
	\end{equation}
	Next we define the $M\times M$ matrices 
	\begin{equation*}
		A_z:=(\alpha_{z,\ell,\mathbf{j}})_{\substack{1\le \ell\le M\\ \mathbf{j}\notin I_z, |\mathbf{j}|=m}} \text{\qquad and \qquad} B_z:=(b_{z,{\mathbf i},\ell})_{\substack{1\le \ell\le M\\ \mathbf{i}\notin I_z, |\mathbf{i}|=m}}.
	\end{equation*}
	By comparing \eqref{linear_form_phi} with \eqref{linear_form_phi_2}, we see that  $B_z=c_z A_z^{-1}$.  From now on, we let $c_z:=\det A_z$, which is in $V_{F,G}(M)$.  Then	
	\begin{equation}
		b_{z,{\mathbf i},\ell}\in V_{F,G}(M-1)
	\end{equation}
	for each $|\mathbf{i}|=m$, $\mathbf{i}\notin I_z$ and $1\le \ell\le M$ by Cramer's rule. This comparison also gives
	\begin{equation}\label{abc_relation}
		\begin{split}
			c_z&=\sum_{\ell=1}^M b_{z,{\mathbf i},\ell} \alpha_{z,\ell,\mathbf{i}} \in V_{F,G}(M)  \text{\qquad for each } |\mathbf{i}|=m, \mathbf{i}\notin I_z, \text{ and } \\
			c_zc_{z,\mathbf{j}}&=\sum_{\ell=1}^M b_{z,{\mathbf i},\ell} \alpha_{z,\ell,\mathbf{j}} \in V_{F,G}(M)   \text{\qquad for each } |\mathbf{i}|=m, \mathbf{i}\notin I_z \text{ and for each }\mathbf{j}\in I_z.
		\end{split}
	\end{equation}
	Then from the proof of \cite[Theorem 5.7]{levin2019greatest},  \eqref{log_FG}, \eqref{linear_form},  \eqref{linear_form_phi} and \eqref{abc_relation}, we obtain 
	\begin{equation}
		\begin{split}
			|L_{z,\mathbf{i}}(\Phi(\mathbf{g}(z)))|&\le \log |\mathbf{g}(z)^{\mathbf{i}}|+\max_{\mathbf{j}\in I_z}\{\log|c_z(z)|,\log|c_z(z) c_{z,\mathbf{j}}(z)|\}+O(1)\\
			&\le \log |\mathbf{g}(z)^{\mathbf{i}}|+\max_{\ell}\log |b_{z,{\mathbf i},\ell}(z)| +\log\|F\|_z+\log\|G\|_z+O(1),
		\end{split}
	\end{equation}
	which gives the following key inequality for computing the corresponding Weil functions
	$$
	|L_{z,\mathbf{i}}(\Phi(\mathbf{g}(z)))|-\|L_{z,\mathbf{i}}\|_z\le \log |\mathbf{g}(z)^{\mathbf{i}}|+\log\|F\|_z+\log\|G\|_z+O(1).
	$$
	
	To apply Theorem \ref{movingsmt0}, we first note that the coefficients of $ L_{z,\mathbf{i}} $ are in $V:=V_{F,G}(M)$ by \eqref{abc_relation}.
	Since the set $\{g_0^{i_0}\cdots g_n^{i_n}: i_0+\cdots+i_n=m\}$ is linearly non-degenerate over $V_{F,G}(Mb+1)$, we must have that  $\Phi(\mathbf{g}):\CC\to\PP^{M-1}$ is linearly nondegenerate over $V_{F,G}(Mb) =V(b)$ (as in Theorem \ref{movingsmt0}). 
	We also note that when computing Weil function \eqref{Weil}, we may assume that the coefficients of  $L_{z,\mathbf{i}}$ are entire functions without common zeros.   Then by the construction,     for each $z\in\mathbb C$, the hyperplanes defined by evaluating 
	the linear forms $ L_{z,\mathbf{i}} $ at $z$ with  $ |\mathbf{i}|=m$ and $\mathbf{i}\notin I_z$  are in general position.
	
	The other part of the arguments  is similar to the proof of \cite[Theorem 5.7]{levin2019greatest}, so we omit the details. 
\end{proof}


\begin{proof}[Proof of Theorem \ref{movinggcd}]
	Let $\alpha$ and $\beta$ be one of the nonzero coefficients of $F$ and $G$ respectively. Since $v_{z}^{+}(F({\bf g}))\le v_{z}^{+}(\frac{1}{\alpha}F({\bf g}))+v_{z}^{+}(\alpha)$ and $v_{z}^{+}(G({\bf g}))\le v_{z}^{+}(\frac{1}{\beta}G({\bf g}))+v_{z}^{+}(\beta)$ for each $z\in\mathbb C$, we have 
	\begin{align*}
		N_{S,{\rm gcd}}(F({\bf g}),G({\bf g}))&\le N_{S,{\rm gcd}}(\frac{1}{\alpha}F({\bf g}),\frac{1}{\beta}G({\bf g}))+N_{\alpha}(0,r)+N_{\beta}(0,r)\\
		&\le_{\exc} N_{S,{\rm gcd}}(\frac{1}{\alpha}F({\bf g}),\frac{1}{\beta}G({\bf g}))+ \frac{c_1} \ell \cdot T_{\mathbf{g}}(r) \quad\text{(by Proposition \ref{polyheight})}
	\end{align*}
	for some constant $c_1$.
	Therefore, we will assume that one of the coefficients in each expansion of $F$ and $G$ is 1.
	Recall that $M:=M_{m}:=2\binom{m+n-d}{n}-\binom{m+n-2d}{n}$  and $M':=M'_{m}:=\binom{m+n}{n}-M=O(m^{n-2})$.
	Elementary computations give that 
	\begin{align*}
		\binom {m+n}{n} &=\frac{m^n}{n!}+\frac{(n+1)m^{n-1}}{2(n-1)!}+O(m^{n-2}),\cr
		c_{m,n,d}&=\frac{m^{n+1}}{(n+1)!}+\frac{m^n}{2(n-1)!}+O(m^{n-1}),\cr
		M&= \frac{m^n}{n!}+ O(m^{n-1}). 
	\end{align*}
	Then 
	\begin{align*}
		\frac{m}{n+1}\binom{m+n}{n}-c_{m,n,d}&=O(m^{n-1}).
			\end{align*}

	Let $\epsilon>0$ be given. Due to the above estimates, we may  choose $m=O(\epsilon^{-1})$ 
	so that $m\ge2d$,
	\begin{align}\label{findm} 
		\frac{M'mn}{M}\le\frac{\epsilon}{4},\quad\text{and } \quad \frac1M\left(\frac{m}{n+1}\binom{m+n}{n}-c_{m,n,d}-M'm\right) \le\frac{\epsilon}{4(n+1)}.
	\end{align}
	By \eqref{vlimit} we
	may then choose a sufficiently large integer $b\in\mathbb{N}$ such
	that 
	\begin{align}
		\frac{w}{u}-1\le\frac{\epsilon}{4mn},\label{wu}
	\end{align}
	where $w:=\dim_{{\bf k}}V_{F,G}(Mb)$ and $u:=\dim_{{\bf k}}V_{F,G}(Mb-M)$.
	Suppose that  the set $\{g_0^{i_0}\cdots g_n^{i_n}: i_0+\cdots+i_n=m\}$ is linearly non-degenerate over $V_{F,G}(Mb+1)$.
	Then by Theorem \ref{Mfundamental}, 
	we have 
	\begin{align}\label{gcdest}
		N_{\rm gcd}&(F({\bf g}),G({\bf g}),r) 
		\le_{\rm exc}  
		\left(\frac{M'}{M}mn+(\frac{w}{u}-1)mn   \right)T_{\bf g}(r)\cr
		&+ \frac1M \left(\frac{m}{n+1}\binom{m+n}{n}-c_{m,n,d}-M'm\right)\sum_{i=0}^nN_{g_i}(0,r)+\frac{c_{m,n,d}}{M}  \sum_{i=0}^n N^{(L)}_{ g_i}(0,r)\cr
		& +\frac1M \binom {m+n-2d}{n}N_{\rm gcd}(\{{\bf g}^{\bf i}\}_{{\bf i}\in I},r) 
		+ \frac{c}{M}(T_{ F}(r)+T_{  G}(r)).
	\end{align}
	We note that the assumption  that not both $F$ and $G$ vanish at any of the points of $\{[1:0:\cdots:0],\hdots,[0:\cdots:0:1]\}$ implies that $(d,0, \hdots,0) ,\hdots, (0, \hdots,0,d) \in I$.  As $g_0,\hdots,g_n$ have no common zeros, we have 
	$N_{\rm gcd}(\{{\bf g}^{\bf i}\}_{{\bf i}\in I},r)=0$.  By \eqref{findm}, \eqref{wu} and that 
	\begin{align}\label{NT}
		N_{g_i}(0,r)\le T_{\bf g}(r)\quad\text{ for $0\le i\le n$, }
	\end{align}
	we see that the sum of first two terms of the right hand side of  \eqref{gcdest}  is bounded by  $\frac{3\epsilon}4  T_{\bf g}(r)$.
	By Corollary \ref{MborelCor},  $T_{F}(r)+T_{G}(r)\le_{\exc} \frac{c_2}{\ell} T_{\bf g}(r) $ for some positive constant $c_2$.  In conclusion, we derive from \eqref{gcdest} that 
	\begin{align}\label{gcdest2}
		N_{\rm gcd}(F({\bf g}),G({\bf g}),r) 
		&\le_{\rm exc}  \frac{3\epsilon}4 T_{\bf g}(r) +m \sum_{i=0}^n N^{(L)}_{ g_i}(0,r)    + \frac{ c_3}{M\ell}  T_{\mathbf{g}}(r)\cr
		&\le \frac{3\epsilon}4  T_{\bf g}(r) +\frac{mL}{\ell} \sum_{i=0}^n N_{ g_i}(0,r)    +  \frac{ c_3}{M\ell}  T_{\mathbf{g}}(r)\cr
		&\le (\frac{3\epsilon}4+\frac{m(n+1)L}{\ell}+ \frac{ c_3}{M\ell})   T_{\bf g}(r)\quad	\text{(by \eqref{NT})}
	\end{align}
	where $c_3=c\cdot c_2$.
	Let $N\ge O(\epsilon^{-3})$ be an integer greater than  $4((n+1)mL+c_3/M)\epsilon^{-1}$.  Then for  $\ell\ge N$, we   have
	$$
	N_{\rm gcd}(F({\bf g}),G({\bf g}),r)\le_{\exc} \epsilon T_{\bf g}(r). 
	$$

	Finally, if  the set $\{g_0^{i_0}\cdots g_n^{i_n}: i_0+\cdots+i_n=m\}$ is linearly  degenerate over $V_{F,G}(Mb+1)$, then there exists a homogeneous polynomial $A\in V_{F,G}(Mb+1)[x_0,\dots,x_n]$ of $\deg A=m$  such that $A(g_0,\dots,g_n)=0$. Hence we may apply Corollary \ref{MborelCor}	to 	derive that there exists a non-trivial $n$-tuple of integers $(j_1,\dots,j_n)$ with $|j_1|+\cdots+|j_n|\le 2m$ and a positive real $c_4$ depending only on $A$ such that 
	\begin{equation*}
		T_{(\frac{g_1}{g_0})^{j_1}\cdots (\frac{g_n}{g_0})^{j_n}}(r)\le_{\exc} \frac{c_4}{\ell} T_{\mathbf{g}}(r).
	\end{equation*}
	Since we may enlarge $N$ whenever necessary,  we can conclude the proof by taking  $\ell\ge N\ge c_4\cdot \epsilon^{-3}. $
\end{proof}

\section{Proofs of  Theorem \ref{main_thm_1}}\label{thm1}

\subsection{Polynomials over a ring with logarithmic differentials}
Let $\mathbf g=(g_0,\hdots,g_n)$, where $g_0,\hdots,g_n$ are non-constant entire functions without common zeros.
Recall that 
$$
\mathbb K_{\bf g}:=\{a: a \text{ is a meromorphic function with }T_{a}(r)\le {\rm O} (T_{\mathbf g}(r))\},
$$ which is a field of the same growth w.r.t. ${\bf g}$.    
We note $a'\in  \mathbb K_{\bf g}$ if $a\in  \mathbb K_{\bf g}$. 
Let $u_i=g_i/g_0$, for $0\le i\le n$, and $\mathbf{u}=(1,u_1,\hdots,u_n)$.
Since $T_{u_i}(r)\le T_{\mathbf g}(r)$,  we see that $u_i$, $u_i'$ and $u_i'/u_i$ are all in $ \mathbb K_{\bf g}$.  

Let  $\mathbf{x}:=(x_{0},\ldots,x_{n})$.  For $\mathbf{i}=(i_{0},\ldots,i_{n})\in\mathbb{Z}^{n}$, we let $\mathbf{x^{i}}:=x_{0}^{i_{0}}\cdots x_{n}^{i_{n}}$ and  $\mathbf{u^{i}}:=u_{0}^{i_0}\cdots u_{n}^{i_{n}}$. For a non-constant polynomial $F(\mathbf{x})=\sum_{\mathbf{i}}a_{\mathbf{i}}\mathbf{x}^{\mathbf{i}}\in   \mathbb K_{\bf g}[\mathbf{x}]:=  \mathbb K_{\bf g}[x_0,\dots,x_n]$,
we define 
\begin{align}\label{Duexpression}
	D_{\mathbf{u}}(F)(\mathbf{x}):=\sum_{\mathbf{i}}\frac{(a_{\mathbf{i}}\mathbf{u}^{\mathbf{i}})'}{\mathbf{u}^{\mathbf{i}}}\mathbf{x}^{\mathbf{i}}
	=\sum_{\mathbf{i}}(  a_{\mathbf{i}}'+a_{\mathbf{i}}  \cdot\sum_{j=1}^{n}i_j\frac{u_{j}'}{u_{j}})\mathbf{x}^{\mathbf{i}}\in\mathbb K_{\mathbf g}[\mathbf{x}].
\end{align} 
A direct computation shows that 
\begin{align}\label{fuvalue}
	F(\mathbf{u})'=D_{\mathbf{u}}(F)(\mathbf{u}),
\end{align}
and that  the following product rule 
\begin{align}\label{productrule}
	D_{\mathbf{u}}(FG)=D_{\mathbf{u}}(F)G+FD_{\mathbf{u}}(G)
\end{align}
holds for  $F,G\in \mathbb K_{\mathbf g}[\mathbf{x}]$.
The following lemma is a modification of   \cite[Lemma 3.1]{GSW20}.  

\begin{lemma}\label{coprime}
	Let $F$ be a nonconstant  homogeneous polynomial in $\mathbb K_{\mathbf g}[\mathbf{x}]$  with no monomial factors and no repeated factors.  Then   $F$   and $D_{\mathbf{u}}(F)$
	are coprime  in  $\mathbb K_{\mathbf g}[\mathbf{x}]$ unless there exists a non-trivial tuple of integers $(m_1,\hdots,m_n)$ with $\sum_{i=1}^n|m_i|\le 2\deg F$  such that 
	$ T_{u_1^{m_1} \cdots u_n^{m_n}}\le T_{F}(r).$
\end{lemma}
\begin{proof}
	Let $F=F_1\cdots F_k$, where $F_i$, $1\le i\le k$, are irreducible (homogeneous) polynomials in $\mathbb K_{\mathbf g}[\mathbf{x}]$.
	By \eqref{productrule}, we have 
	$$
	D_{\mathbf{u}}(F)=D_{\mathbf{u}}(F_1)F_2\cdots F_k+\cdots +F_1\cdots F_{k-1}D_{\mathbf{u}}(F_k).
	$$
	If  $F$   and $D_{\mathbf{u}}(F)$ are not coprime in $ \mathbb K_{\mathbf g}[\mathbf{x}]$, then the irreducible common factors  are among the irreducible polynomials  $F_1,\dots,F_k\in  \mathbb K_{\mathbf g}[\mathbf{x}]$.   Assume that $F_j$, $1\le j\le k$, is a common factor.  Then  $D_{\mathbf{u}}(F_j) $ is divisible by $F_j$.
	Write $F_j=\sum_{\mathbf{i}}b_{\mathbf{i}}\mathbf{x}^{\mathbf{i}}\in  \mathbb K_{\mathbf g}[\mathbf{x}]$, which contains at least two distinct terms, since $F$ has no monomial factors. 
	Then   $D_{\mathbf{u}}(F_j) /F_j$ is a non-zero constant in $\mathbb K_{\bf g}$  since  $D_{\mathbf{u}}(F_j)$ is not zero and $\deg D_{\mathbf{u}}(F_j)=\deg F_j$.
	Comparing the coefficients of $F_j$ and $D_{\mathbf{u}}(F_j)$, for nonzero $b_{\mathbf{i}}$ and $b_{\mathbf{j}}$   in $ \mathbb K_{\bf g} $ ($\mathbf{i}\ne\mathbf{j}$), we have from \eqref{Duexpression} that
	\begin{equation}\label{eq_diff1}
		\frac{(b_{\mathbf{i}}\mathbf{u}^{\mathbf{i}})'}{b_{\mathbf{i}}\mathbf{u}^{\mathbf{i}}}=\frac{(b_{\mathbf{j}}\mathbf{u}^{\mathbf{j}})'}{b_{\mathbf{j}}\mathbf{u}^{\mathbf{j}}},
	\end{equation} 
	where $\mathbf{i}=(i_0,i_1,\hdots,i_n)$, $\mathbf{j}=(j_0,j_1,\hdots,j_n)$ and $\sum_{h=0}^n i_h=\sum_{h=0}^n j_h=\deg F_j$.
	It implies that 
	$\frac {b_{\mathbf{i}}\mathbf{u}^{\mathbf{i}}}{b_{\mathbf{j}}\mathbf{u}^{\mathbf{j}}}\in\mathbb C$.  Therefore,
	$$
	T_{u_1^{i_1-j_1} \cdots u_n^{i_n-j_n}}(r)\le T_{F_j}(r)\le T_{F}(r).
	$$
\end{proof}

\subsection{A General Theorem}
For convenience of discussion, we denote by $\mathcal E$   the collection of entire functions.  For a positive integer $\ell$, we let
\begin{align}\label{Em}
	\mathcal E_{\ell}:=\{g\in\mathcal E\setminus \{0\}\,:\, \text{the zero multiplicity of $g$ at each $z\in \mathbb C$ is at least $\ell$ if $g(z)=0$}\}.
\end{align}
It's clear that $f\cdot g\in \mathcal E_{\ell}$ if $f, g\in \mathcal E_{\ell}$.  

We first show the following.
\begin{theorem}\label{main_thm_1G}
	Let $G$ be a non-constant  homogeneous polynomial   in $\mathbb C[x_0,\hdots,x_n]$ with no monomial factors and no repeated factors.  Assume that the hypersurface  defined by   $G$ in $\mathbb P^n$ and the coordinate hyperplanes are in general position. Let $g_0, g_1,\dots,g_n$ be  entire functions in  $\mathcal E_{\ell}$  with no common zeros.  Let  $\mathbf{g}=(g_0,g_{1},\ldots,g_{n}):\mathbb C\to \mathbb P^{n}$.   
	Then for any   sufficiently small $\epsilon >0$, there exist positive integers   $\ell_1\ge O(\epsilon^{-3})$ and $\ell_2=O(\epsilon^{-1})$  independent of  $\mathbf{g}$
	such that 	 if $\ell\ge\ell_1$, then either there is a non-trivial $n$-tuple $(i_1,\hdots,i_n)$ of integers with  $\sum_{j=1}^n |i_j|\le \ell_2$ such that
	\begin{align}\label{exception0}
		T_{(\frac{g_1}{g_0})^{i_1}\cdots (\frac{g_n}{g_0})^{i_n}}\le_{\rm exc} \epsilon^3  T_{\mathbf{g}}(r), 
	\end{align}
	or the following holds.
	\begin{itemize}
		\item[\rm (i) ]  
		$N_{G(\mathbf{g})}(0,r)-N^{(1)}_{G(\mathbf{g})}(0,r)\le_{\exc} \epsilon T_{\mathbf{g}}(r)$, and 
		\item[\rm (ii) ]  
		$N^{(1)}_{G(\mathbf{g})}(0,r)\ge_{\rm exc}  (\deg  G- \epsilon)\cdot T_{\mathbf{g}}(r)$.
	\end{itemize} 
\end{theorem}
\begin{proof} 
	We first note that if $G(\mathbf{g})=0$, then  we have \eqref{exception0} following Corollary \ref{MborelCor}  if  $\ell\ge\ell_1\ge c\epsilon^{-3}$ for some positive constant $c$.  Therefore, it suffices to consider (i) when $G(\mathbf{g})$ is not identically zero.
	Let $u_i=g_i/g_0$, $0\le i\le n$ and $\mathbf u=(1,u_1,\hdots,u_n)$.  We may assume that each $u_i$, $1\le i\le n$, is not constant, otherwise $T_{u_i}(r)=O(1)$.
	Let $d=\deg G$.
	Then $G(\mathbf{g})=g_0^d G(\mathbf u)$, $D_{\mathbf u}(G)(\mathbf{g})=g_0^d D_{\mathbf u}(G)(\mathbf{u})$  and 
	\begin{align}\label{Dug}
		G(\mathbf{g})'= d g_0'  g_0^{d-1}G(\mathbf{u}) +g_0^dG(\mathbf u)'=d\frac{g_0'}{g_0}G(\mathbf{g}) +D_{\mathbf u}(G)(\mathbf{g}) 
	\end{align}
	by  \eqref{fuvalue}.
	
	Let $z_0\in\mathbb C$.  If $v_{z_0}( G(\mathbf{g}))\ge 2$, 
	it follows from \eqref{Dug} that $v_{z_0}( D_{\mathbf{u}}(G)(\mathbf{g}))\ge v_{z_0}( G(\mathbf{g}))-1$ since $v_{z_0}( G(\mathbf{g})')=v_{z_0}( G(\mathbf{g}))-1$ and $\frac{g_0'}{g_0}$ has only at worst simple poles.
	Hence, 
	$$
	\min\{v_{z_0}^+(G(\mathbf{g})),v_{z_0}^+(D_{\mathbf{u}}(G)(\mathbf{g}))\}\ge v_{z_0}^+( G(\mathbf{g}))-\min\{1,v_{z_0}^+( G(\mathbf{g}))\} 
	$$
	for any $z_0\in\mathbb C$.  Consequently,
	\begin{align}\label{truncate}
		N_{\gcd}(G(\mathbf{g}), D_{\mathbf{u}}(G)(\mathbf{g}),r)\ge N_{G(\mathbf{g})}(0,r)-N^{(1)}_{G(\mathbf{g})}(0,r).
	\end{align}
	On the other hand, by Lemma \ref{coprime}, $G$ and $D_{\mathbf{u}}(G)$ are coprime in $ \mathbb K_{\bf g}[\bf x] :=\mathbb K_{\bf g} [x_0,\hdots,x_n]$ or \eqref{exception0} holds with $\sum_{j=0}^n |i_j|\le 4d$.   Since the coefficients of $G$ are in $\mathbb C$, we have that
	$D_{\mathbf{u}}(G)(\mathbf{g})\in \mathbb C(\frac{u_1'}{u_1},\hdots,\frac{u_n'}{u_n})[\bf x] $ and that  $G$  and $D_{\mathbf{u}}(G)$ are coprime in $\mathbb C(\frac{u_1'}{u_1},\hdots,\frac{u_n'}{u_n})[\bf x] $.
	Furthermore, since $[G=0]$   and the coordinate hyperplanes are in general position, we see that $G(P)\ne 0$ for 
	$P\in \{(1,0,\hdots,0),\hdots,(0,\hdots,0,1)\}$.
	Therefore, we can apply Theorem \ref{movinggcd} to obtain that for any  sufficiently small  $\epsilon>0$,  there exist  integers $\ell_1\ge O(\epsilon^{-3})$ and  $\ell_2=O(\epsilon^{-1})$  such that either \eqref{exception0} holds, or
	\begin{align*} 
		N_{\rm gcd} (G({\mathbf{g}}),D_{\mathbf{u}}(G)(\mathbf{g}),r)\le_{\rm exc}   \epsilon  T_{\bf g}(r).
	\end{align*}
	Then \eqref{truncate} implies that
	\begin{align}\label{truncate7}
		N_{G(\mathbf{g})}(0,r)-N^{(1)}_{G(\mathbf{g})}(0,r) \le_{\exc}   \epsilon  T_{\bf g}(r).
	\end{align}
	This proves the first assertion.

	Since $[G=0]$   and the coordinate hyperplanes are in general position,   it follows from Theorem \ref{SMTmoving} with polynomials $ G $, $x_0,\hdots,x_n$  that there exist positive integers $M\ge O(\epsilon^{-2})$ and $D=O(\epsilon^{-1})$ depending only on $\epsilon$ and $\deg G$ such that either the image of $\mathbf{g}$ is contained in a hypersurface in $\mathbb P^n(\mathbb C)$ of degree bounded by $D$, which again imply \eqref{exception0} by Corollary \ref{MborelCor} (enlarge $\ell_1$ if necessary), or
	\begin{equation*}
		(1-\frac{\epsilon}{2d})T_{\mathbf{g}}(r)\le_{\exc} \sum_{i=0}^n N_{g_i}^{(M)}(0,r)+\frac{1}{d}N_{G(\mathbf{g})}(0,r).
	\end{equation*}
	Hence 
	\begin{equation}\label{applySMT}
		\begin{split}
			N_{G(\mathbf{g})}(0,r)&\ge_{\exc} d(1- \frac{\epsilon}{2d} )  T_{\mathbf{g}}(r) - \frac{dM}{\ell}\sum_{i=0}^n N_{g_i}(0,r)\\
			&\ge (d- \frac{\epsilon}{2})  T_{\mathbf{g}}(r)-\frac{dM(n+1)}{\ell} T_{\mathbf{g}}(r).
		\end{split}
	\end{equation}
	Thus, we obtain the second assertion by taking $\ell_1\ge 2dM(n+1)\epsilon^{-1}$.
\end{proof}

\subsection{Proof of Theorem \ref{main_thm_1}}
The proof of Theorem \ref{main_thm_1} is the combination of Theorem \ref{main_thm_1G} for $n=2$ and the following proposition.

\begin{proposition}\label{n2}
	Let $G\in \mathbb C[x_0,x_1,x_2]$ be a non-constant  homogeneous polynomial   with no monomial factors and no repeated factors. 
	Assume that the plane curve  $[G=0]$  and $H_i=[x_{i}=0]$, $0\le i\le 2$, are in general position.  
	Let $\epsilon$ be  a sufficiently small positive real and $n_{1}, n_{2}$ be integers not both zeros such that $|n_i|\le O(\epsilon^{-1})$.  Then there exists a proper Zariski closed subset $W\subset\mathbb P^2(\mathbb C)$ and an effectively computable positive integer $\ell\ge O(\epsilon^{-3})$ such that for any nonconstant   orbifold entire curve  $\mathbf{g}=(g_0,g_1,g_2):\mathbb C\to (\mathbb P^{2}, \Delta)$, where 
	$\Delta=(1-\frac 1{m_1})H_0+(1-\frac 1{m_2})H_1+(1-\frac 1{m_3})H_{2}$  with $g_0\ne 0$,  $m_1,m_2,m_3\ge \ell$, and
	\begin{align}\label{ht00}
		T_{(\frac{g_1}{g_0})^{n_{1}}  (\frac{g_2}{g_0})^{n_{2}}}(r)\le \epsilon^3T_{\mathbf{g}}(r),
	\end{align}
	we have the following two inequalities  
	\begin{itemize}
		\item[\rm (i) ]  
		$N_{G(\mathbf{g})}(0,r)-N^{(1)}_{G(\mathbf{g})}(0,r)\le_{\exc} \epsilon T_{\mathbf{g}}(r)$, and 
		\item[\rm (ii) ]   
		$N^{(1)}_{G(\mathbf{g})}(0,r)\ge_{\rm exc}  (\deg  G- \epsilon)\cdot T_{\mathbf{g}}(r)$,
	\end{itemize} 
	if the image of $\mathbf{g}$ is not contained in $W$.  Furthermore, the exceptional set $W$ is a finite union of closed subsets of the following types: $[x_1^{n_1}x_2^{n_2}=\beta x_0^{n_1+n_2}]$, if $n_1\ge 0$, and $[x_2^{n_2}x_0^{-n_1-n_2}=\beta  x_1^{-n_1}]$, if $n_1< 0$.
\end{proposition}

\begin{remark}
	For each pair $(n_1,n_2)$, the proper Zariski closed subset $W$ can be constructed explicitly.
	Moreover,  we have only finitely many choices of $n_1$ and $n_2$ for a given $\epsilon$ when apply Theorem \ref{main_thm_1G}. 
\end{remark}

\begin{proof}[Proof of Theorem \ref{main_thm_1}]
	Let $\epsilon>0$.  We may assume that $\epsilon$ is sufficiently small.  We first note that we only need to consider the nonconstant holomorphic map $\mathbf{g}=(g_0,g_{1}, g_{2}):\mathbb C\to \mathbb P^{2}$ with $g_i$, $0\le i\le 2$, not identically zero by including the coordinate line $[x_i=0]$, $0\le i\le 2$, to $W$.
	By Theorem \ref{main_thm_1G}, we find  positive integers $\ell_1\ge O(\epsilon^{-3})$ and  $\ell_2=O(\epsilon^{-1})$  independent of $\mathbf{g}$  such that for any nonconstant  holomorphic map  $\mathbf{g}=(g_0,g_{1}, g_{2}):\mathbb C\to \mathbb P^{2}$, where 
	$g_0, g_1,g_2\in \mathcal E_{\ell_1}$ with no common zeros, we have either both (i) and (ii) are valid or \eqref{ht00} in Proposition \ref{n2} holds for some integers $n_1$ and $n_2$ whose absolute values are bounded by $\ell_2$.   We then apply Proposition \ref{n2} for the latter situation for each possible pair of $(n_1,n_2)$ to conclude the proof.
\end{proof}	

\subsection{Proof of Proposition \ref{n2}}
To prove Proposition \ref{n2}, we need the following version of Hilbert Nullstellensatz reformulated from \cite[Chapter~XI]{waerden1967} or \cite[Chapter~IX, Lemma 3.7]{langalgebra}. 
\begin{proposition} \label{HilbertN}
	Let $A$ be a ring.  Let $\{Q_i\}_{i=1}^{n+1}$ be a set of homogeneous polynomials in $ A [x_0,\dots,x_n]$ such that their zero locus are in  general position.   Then there exists a positive integer $s$, $R\in A\setminus\{0\} $ and  $P_{ji}\in A[x_0,\dots,x_n]$, $1\le i,j\le n+1$,  such that 
	$$
	x_j^s\cdot R=\sum_{i=1}^{n+1} P_{ji} Q_i
	$$
	for each $0\le j\le n.$
\end{proposition}

\begin{proof}[Proof of Proposition \ref{n2}]
	Suppose that  \eqref{ht00} holds.   Let $u_1=g_1/g_0$ and $u_2=g_2/g_0$.   We may assume that  $n_1$ and $n_2$ are coprime and let $u_{1}^{n_{1}}  u_{2}^{n_{2}}=\lambda$.  Then there exist integers $a$ and $b$ such that $n_1a+n_2b=1$ and 
	\begin{align}\label{Tlambda}
		T_{\lambda}(r)\le \epsilon^3  T_{\mathbf{g}}(r).
	\end{align}
	Since $T_{\lambda}(r)=T_{\lambda^{-1}}(r)$, we may exchange the sign of $n_1$ and $n_2$ simultaneously.
	Moreover, we can also rearrange the indices  of $u_1$ and $u_2$.  
	Therefore, we may assume that 
	$n_2\ge  n_1\ge 0$ if $n_1n_2\ge 0$   and $0< n_2\le -n_1$ if $n_1n_2<0$.    It's clear that $n_2>0$ in this setting as $(n_1,n_2)\ne (0,0)$.
	For the choice of $a$ and $b$, we note that if $n_1=0$, then $n_2=1$ and we can simply take $a=0$ and $b=1$.  If $n_1\ne 0$,  
	we can assume that $0< b \le |n_1|$ since $n_1(a+kn_2)+n_2(b-kn_1)=1$ for any integer $k$.  Then we have $|a|<n_2$.  We can further verify that $a<b$ for all cases.
	
	Let $\beta=u_1^bu_2^{-a}$.  We may write
	\begin{align}\label{u12b}
		u_1=\lambda^a \beta^{n_2}\qquad\text{and}\qquad  u_2=\lambda^b\beta^{-n_1}.
	\end{align}
	Then $T_{\mathbf{g}}(r)=T_{[1:u_1:u_2]}(r)= T_{[1:\lambda^a \beta^{n_2}:\lambda^b\beta^{-n_1}]}(r).$
	As we have set $n_2>0$, it suffices to consider for $n_1\ge 0$ and $n_1<0$.
	From the definition of the characteristic function and that $|a|+|b|<|n_1|+|n_2|$, we have 
	\begin{align}\label{Tgbeta1}
		T_{\mathbf{g}}(r) = T_{[1:\lambda^a \beta^{n_2}:\lambda^b\beta^{-n_1}]}(r) \le \max\{|n_1|,|n_2|\} T_{\beta} (r)+(|n_1|+|n_2|) T_{\lambda}(r)  
	\end{align}
	if  $n_1< 0$;
	and 
	\begin{align}\label{Tgbeta2}
		T_{\mathbf{g}}(r) = T_{[\lambda^{-b}\beta^{n_1}:\lambda^{a-b} \beta^{n_1+n_2}:1]}(r) \le ( n_1+n_2 ) T_{\beta} (r)+( n_1 + n_2 ) T_{\lambda}(r)  
	\end{align}
	if   $n_1\ge0$.
	Let 
	\begin{align}\label{LamG}
		\Lambda=X^{n_{1}}  Y^{n_{2}}\qquad\text{and}\qquad  T=X^bY^{-a}
	\end{align} 
	be two variables.  Then
	\begin{align}\label{XY}
		X=\Lambda^a T^{n_2}\qquad\text{and}\qquad  Y=\Lambda^bT^{-n_1}.
	\end{align} 
	Let $G_1(X,Y)=G(1,X,Y)$.
	Let $B_{\Lambda}(T)  \in \mathbb C[\Lambda,\Lambda^{-1}][T]$  be the polynomial such that $B_{\Lambda}(0)\ne 0$ and 
	\begin{align}\label{GB0}
		G_1(X,Y)=G_1(\Lambda^a T^{n_2},\Lambda^b T^{-n_1})= T^{M_1}B_{\Lambda}(T) 
	\end{align}
	for some integer $M_1$.
	Let $B(\Lambda,T)\in \mathbb C[\Lambda,T]$ be the polynomial such that   $B(0,T)\ne 0$ and
\begin{align}\label{GB1}
	B(\Lambda,T)=\Lambda^{M_2} B_{\Lambda}(T) 
	\end{align}
	for some integer $M_2$.	
Then 
	\begin{align}\label{GB}
		G_1(X,Y)=G_1(\Lambda^a T^{n_2},\Lambda^b T^{-n_1})=T^{M_1}B_{\Lambda}(T) =T^{M_1}\Lambda^{M_2}B(\Lambda,T),
	\end{align}
	 and  $B(\Lambda,0)\in \mathbb C[\Lambda]$ is not identically zero.  Therefore, there are at most finite	$ \gamma_1,\dots,\gamma_s\in   \mathbb{C} \text{  such that } B(\gamma_i,0)=0 \text{ for }1\le i\le s.$

	We note that  $B(\Lambda,T)$ cannot be constant as $G$ has no monomial factors.
	We now claim that $B(\Lambda,T)\in \mathbb C[\Lambda,T]$ is square free.
	We first rewrite \eqref{GB} as  
	$$
	G_1(X,Y)=X^{bM_1+n_1M_2}Y^{-aM_1+n_2M_2}B(X^{n_{1}}  Y^{n_{2}},X^bY^{-a}).
	$$
	If  $B_0(\Lambda,T)$  is an irreducible factor of  $B(\Lambda,T)$   in $ \mathbb C[\Lambda,T]$, then $B_0(X^{n_{1}}  Y^{n_{2}},X^bY^{-a})= X^{\ell_1}Y^{\ell_2}H(X,Y)$, where $H(X,Y)\in \mathbb C[X,Y]$,  $H(X,0)\ne 0$, $H(0,Y)\ne 0$. 
	If  $H(X,Y) $ is a constant $\alpha$, then $B_0(\Lambda,T)=B_0(X^{n_{1}}  Y^{n_{2}},X^bY^{-a})= \alpha X^{\ell_1}Y^{\ell_2}$.   We may express $B_0(\Lambda,T)=\sum_{\mathbf i} a_{(i_1,i_2)} \Lambda^{i_1} T^{i_2}$.  Then the above equation implies that $n_1i_1+bi_2=\ell_1$ and $n_2i_1-ai_2=\ell_2$ and hence $(i_1,i_2)=(a\ell_1+b\ell_2, n_2\ell_1-n_1\ell_2)$ if $a_{(i_1,i_2)}\ne 0.$  This implies that   $B_0(\Lambda,T)$ is a monomial, which is not possible.
Therefore, $\deg H\ge 1$.   Then $H$ is not a monomial and it is a non-constant factor of $G$. Since  $G$ is square-free, we conclude that $B_0^2(\Lambda,T)$ is not a factor of $B(\Lambda,T)$. 
	In conclusion,   $B(\Lambda,T)\in \mathbb C[\Lambda,T]$  is square free.
		
	Since $B(\Lambda,T)$ is square free,   the resultant $R(B_{\Lambda},B_{\Lambda}')$ of $B_{\Lambda}$ and $B_{\Lambda}'(T)$ is a Laurent polynomial in  $\mathbb C[\Lambda,\Lambda^{-1}]$, not identically zero.
	Let 
	\begin{align}\label{zeroresultant}
		\text{$\alpha_i$, $1\le i\le t$, be the zeros of  the resultant $R(B_{\Lambda},B_{\Lambda}')$.}
	\end{align}  It is clear that  $\alpha_i\in\mathbb C$.
	Let $B(T):=B_{\lambda}(T)\in \mathbb C[\lambda,\lambda^{-1}][T]$, the specialization of  $B_{\Lambda}(T)$ at $\Lambda=\lambda$.  Then $B(T)$  has no multiple factors in $\mathbb C[\lambda,\lambda^{-1}][T]$ if $\lambda\ne\alpha_i$ for any $1\le i\le t$.

	From now on, we assume that $\lambda\ne\alpha_i$ for any $1\le i\le t$  and $\lambda\ne\gamma_j$ for any $1\le j\le s$. 
	Let $\tilde B\in \mathbb C(\lambda )[Z,U]$ be the homogenization of $B$, i.e. $\tilde B(1,T)=B(T)$.
	Let $D_{(1,\beta)}(\tilde B)   \in \mathbb C(\lambda,\lambda', \frac {\beta'}{\beta})[Z,U]$ be as defined in \eqref{Duexpression} with ${\mathbf u}=(1,\beta)$. 
	By Lemma \ref{coprime}, $\tilde B$ and $D_{(1,\beta)}(\tilde B)$ are coprime homogeneous polynomials in  $\mathbb C(\lambda,\lambda', \frac {\beta'}{\beta})[Z,U]$ unless  there exists a non-zero integer $k$ such that 
	$T_{\beta^k}(r)\le T_{ B}(r)$, which implies    
	\begin{align}\label{Bbeta0}
		T_{\beta}(r)\le T_{  B}(r)\le (|a|+|b|)\deg G \cdot T_{\lambda}(r) \le(|n_1|+|n_2|)\deg G \cdot T_{\lambda}(r).
	\end{align}
	Together with \eqref{Tgbeta1} and \eqref{Tgbeta2}, we have 
	\begin{align}\label{Tgbeta3}
		T_{\mathbf{g}}(r)  &\le (|n_1|+|n_2|) ( T_{\beta}(r)+T_{\lambda}(r) )\cr
		&\le (|n_1|+|n_2|)((|n_1|+|n_2|)\deg G+1) T_{\lambda}(r)  \cr
		&\le 2(|n_1|+|n_2|)^2\deg G  \cdot \epsilon^3 T_{\mathbf{g}}(r) 
	\end{align}
	by  \eqref{Tlambda}.
	This is not possible since $|n_1|+|n_2|\le O(\epsilon^{-1}).$ 
	Therefore, we will assume that $\tilde B$ and $D_{(1,\beta)}(\tilde B)$ are coprime homogeneous polynomials in  $\mathbb C(\lambda,\lambda',\frac {\beta'}{\beta})[Z,U]$.

	Let $\beta=\beta_1/\beta_0$, where $\beta_0$ and $\beta_1$ are entire functions without common zeros.  Moreover, we have $a<b$ and $b>0$ in our setting.  
	Then by \eqref{fuvalue} we have 
	\begin{align}\label{Dbeta}
		\tilde B(\beta_0,\beta_1)'=(\beta_0^{\deg B} B(\beta) )'=\deg B\cdot\frac{\beta_0'}{\beta_0} \tilde B(\beta_0,\beta_1)+     D_{(1,\beta)}(\tilde B)(\beta_0,\beta_1);
	\end{align}
	and   by \eqref{GB} we have
	\begin{equation}\label{gtildeB1}
		G(\mathbf{g}):=G(g_0,g_1,g_2)=g_0^dG_1(u_1,u_2)=g_0^d\beta^{M_1} B(\beta)=g_0^d\beta^{M_1}\beta_0^{-\deg B}\tilde B(\beta_0,\beta_1).
	\end{equation}
	
	Next, we  prove the following.

	\noindent{\bf Claim}. There exists a proper Zariski closed set $W_1$ of $\mathbb P^2(\mathbb C)$, independent of $\mathbf{g}$ such that
	\begin{align}\label{claim}
		N_{G(\mathbf{g})}(0,r)-N_{G(\mathbf{g})}^{(1)}(0,r)\le_{\exc} N_{\gcd} (\tilde B(\beta_0,\beta_1), D_{(1,\beta)}(\tilde B)(\beta_0,\beta_1),r),\end{align}
	if the image of $\mathbf{g}$ is not contained in $W_1$.
	
	The condition that $[G=0]$ is in general position with the coordinate hyperplanes of $\mathbb P^2$  implies that  $G_1(X,Y):=G(1,X,Y)$ can be expanded as 
	\begin{align}\label{expandG}
		G_1(X,Y)=a_0+a_1X^d+a_2Y^d+\cdots \in \mathbb C[X,Y],
	\end{align}
	where $a_i\ne 0$, $0\le i\le 2$.   Then  
	\begin{align}\label{GB2}
		G_1(\lambda^a T^{n_2},\lambda^b T^{-n_1})=a_0+a_1\lambda^{ad} T^{n_2d}+a_2\lambda^{bd} T^{-n_1d}+\cdots.
	\end{align} 
	
	We note that if  $G(\mathbf{g}(z_0))=0$ and $g_0(z_0)=0$, then $g_i(z_0)\ne 0$ for $z_0\in\mathbb C$, and $i=1,2$.
	This can be seen easily. For example, if $G(\mathbf{g}(z_0))=g_0(z_0)=g_1(z_0)=0$, then $g_2(z_0)\ne 0$ and $G(0,0,1)=0$,  which contradicts the assumption that  $[G=0]$ and the coordinate hyperplanes of $\mathbb P^2$ are in general position.
	
	We first consider when  $n_1<0$ and $n_2\ne-n_1$.    Then $ B(T)=G_1(\lambda^a T^{n_2},\lambda^b T^{-n_1})$, which   is a polynomial in $\mathbb C[\lambda,\lambda^{-1}][T]$ of degree $d\cdot\max\{ -n_1,n_2\}$. 
	By \eqref{gtildeB1}, we have
	\begin{equation}\label{gtildeB0}
		G(\mathbf{g}) =g_0^d B(\beta)=g_0^d \beta_0^{-\deg B}\tilde B(\beta_0,\beta_1).
	\end{equation}
	To show \eqref{claim}, it suffices to consider when $v_{z_0}( G(\mathbf{g}))\ge 2$.  
	In this case, we have 
	$$
	v_{z_0}(D_{(1,\beta)}(\tilde B)(\beta_0,\beta_1))\ge v_{z_0}(\tilde B(\beta_0,\beta_1))-1\ge v_{z_0}( G(\mathbf{g}))-1
	$$
	by \eqref{Dbeta}; and
	$$
	v_{z_0}(G(\mathbf{g}) )\le v_{z_0}(\tilde B(\beta_0,\beta_1))+v_{z_0}^+(g_0^d \beta_0^{-\deg B})
	$$
	by \eqref{gtildeB0}.
	We will need to show that  $v_{z_0}^+(g_0^d \beta_0^{-\deg B})=0$ if $v_{z_0}( G(\mathbf{g}))\ge 2$.
	This is clear if $v_{z_0}(g_0)=0$.  Therefore, we assume in addition that $v_{z_0}(g_0)>0$.  
	Then  $v_{z_0}(g_1)=v_{z_0}(g_2)=0$ in this case as noted before.  Consequently, $v_{z_0}(\beta_0)=(b-a)v_{z_0}(g_0) \ge v_{z_0}(g_0)$,
	since $\beta=u_1^bu_2^{-a}=g_1^bg_2^{-a}g_0^{a-b}$ and $a<b$. 
	As  $\deg B=d\cdot\max\{ -n_1,n_2\} \ge d$, we conclude that $v_{z_0}^+ (g_0^d \beta_0^{-\deg B})=0$.
	This shows our claim  \eqref{claim} for this case.
	
	Next, we consider when  $n_1<0$ and $n_2= -n_1$.  Since we assume that $n_1$ and $n_2$ are coprime, we have $n_1=-1$ and $n_2=1$ and  $u_{1}^{-1}  u_{2} =\lambda$.  Therefore, we can simply consider the pair
	$(u_1,\lambda u_1)=(u_1,u_2)$, i.e. taking $a=0$, $b=1$ and $\beta=u_1$.  Consider the following expansion of $G$
	\begin{align}\label{expandG2}
		G_1(X,Y)=\tilde G_d(X,Y)+\cdots +\tilde G_1(X,Y)+a_0 \in \mathbb C[X,Y],
	\end{align}
	where $G_i(X,Y)$ is a homogeneous polynomial of degree $i$ in $\mathbb C[X,Y]$ and $a_0\ne 0$.
	Expand 
	\begin{align}\label{alpha}
		\tilde G_d(X,Y)=(X-\delta_1Y)\cdots (X-\delta_dY).
	\end{align}  Then 
	$G_d(u_1,\lambda u_1)\ne 0$ and $B(T)=G_1(T,\lambda T)$ is a polynomial of degree $d$ with $B(0)=a_0\ne 0$  if $\lambda\ne \delta_i$, $1\le i\le d$.  Then
	\begin{equation}\label{gtildeB}
		G(\mathbf{g}) =g_0^dG _1(u_1,u_2)=g_0^d B(\beta)=g_0^d \beta_0^{-d} \tilde B(\beta_0,\beta_1).
	\end{equation}
	Therefore the claim of \eqref{claim} holds similarly in this case.

	Finally, it remains to consider $n_1\ge 0$.  Then it follows from \eqref{GB2} that $G_1(\lambda^a T^{n_2},\lambda^b T^{-n_1})=T^{-n_1d}\cdot  B(T)$, where $ B(T)$ is a polynomial of degree $(n_1+n_2)d$ with $ B(0)\ne 0$.  Therefore,
	\begin{equation}\label{gtildeB2}
		G(\mathbf{g}) =  g_0^d\beta^{-n_1d} B(\beta)=g_0^d\beta^{-n_1d}\beta_0^{-(n_1+n_2)d}\tilde B(\beta_0,\beta_1)=g_0^d \beta_0^{- n_2 d}\beta_1^{- n_1 d}\tilde B(\beta_0,\beta_1).
	\end{equation}
	Similar to the previous arguments, it suffices to show that $v_{z_0}^+(g_0^d \beta_0^{- n_2 d}\beta_1^{- n_1 d})=0$ if $v_{z_0}( G(\mathbf{g}))\ge 2$ and $v_{z_0}(g_0)>0$.  This can be done as $v_{z_0}(g_1)=v_{z_0}(g_2)=0$ and  $v_{z_0}(\beta_0)=(b-a)v_{z_0}(g_0) \ge v_{z_0}(g_0)$.

	By \eqref{claim}, to complete the proof of (i), 
	it remains to show the following:
	\begin{equation}\label{gcd_GB}
		N_{\gcd} (\tilde B(\beta_0,\beta_1), D_{(1,\beta)}(\tilde B)(\beta_0,\beta_1),r)\le_{\exc} \epsilon T_{\mathbf g}(r). 
	\end{equation}
	Let $A:=\mathbb C[\lambda,\lambda^{-1}, \lambda',\frac {\beta'}{\beta}]$.
	Since $\tilde B$ and $D_{(1,\beta)}(\tilde B)$ are coprime homogeneous polynomials in  $A[Z,U]$, we may apply  Proposition \ref{HilbertN} to find an integer $s$, $R\in A\setminus\{0\} $ and  $F_1,F_2, P_1,P_2\in  A[Z,U]$ such that 
	\begin{align}\label{resul}
		Z^s\cdot R=F_1\tilde B+F_2D_{(1,\beta)}(\tilde B) \quad\text{and }\quad   U^s\cdot R=P_1\tilde B+P_2D_{(1,\beta)}(\tilde B). 
	\end{align}
	As the coefficients of $F_1,F_2, P_1,P_2$ are in $A=\mathbb C[\lambda,\lambda^{-1},\lambda',\frac {\beta'}{\beta}]$,
	by multiplying  some entire functions $\eta_1$ and $\eta_2  $ on the both sides of the equations of   \eqref{resul},  we may assume that the coefficients of $F_i$ and  $P_i$ are entire functions.  Then by evaluating \eqref{resul} at $(\beta_0,\beta_1)$,  we have
	\begin{align}\label{resulevaluating }
		\beta_0^s\cdot R\eta_1 &=F_1(\beta_0,\beta_1)  \tilde B(\beta_0,\beta_1)+F_2(\beta_0,\beta_1) D_{(1,\beta)}(\tilde B) (\beta_0,\beta_1),\cr
		\beta_1^s\cdot R\eta_2&=P_1(\beta_0,\beta_1) \tilde B(\beta_0,\beta_1)+P_2(\beta_0,\beta_1) D_{(1,\beta)}(\tilde B) (\beta_0,\beta_1). 
	\end{align}
	Since $\beta_0$ and $\beta_1$ have no common zeros, we see that 
	\begin{equation}\label{min_ord}
		\min\{v_z^+ (\tilde B(\beta_0,\beta_1)), v_z^+(D_{(1,\beta)}(\tilde B) (\beta_0,\beta_1))\} \le v_z^+(R) +v_z^+(\eta_1)+v_z^+(\eta_2)
	\end{equation}
	for each $z\in\CC$. 
	Therefore,
	\begin{equation}\label{gcd_GB3}
		N_{\gcd} (\tilde B(\beta_0,\beta_1), D_{(1,\beta)}(\tilde B)(\beta_0,\beta_1),r)\le N_R(0,r)+  N_{\eta_1 }(0,r )+N_{\eta_2 }(0,r ).
	\end{equation}
	
	Since $\beta=u_1^bu_2^{-a}=g_1^bg_2^{-a}g_0^{a-b}$,
	we have $\frac {\beta'}{\beta}=b\frac {g_1'}{g_1}-a\frac {g_2'}{g_2}+(a-b)\frac {g_0'}{g_0}$.  Hence
	$$
	T_{\frac {\beta'}{\beta}}(r)\le \sum_{i=0}^2 T_{\frac{g_i'}{g_i}}(r)+O(1)\le_{\exc} \frac{3}{\ell}  T_{\mathbf g}(r) +O(1) \le \epsilon^3   T_{\mathbf g}(r)
	$$
	by Proposition \ref{coutingzero} and taking $\ell>  3   \epsilon^{-3}$.
	We also note that  
	$T_{\lambda}(r)\le \epsilon^3  T_{\mathbf{g}}(r)$ and $T_{\lambda'}(r)\le \epsilon^3  T_{\mathbf{g}}(r)$.
	Therefore,  for any $\alpha\in A$, we have $T_{\alpha}(r)\le c_{\alpha}\epsilon^3  T_{\mathbf{g}}(r)$, where $c_{\alpha}$ is a positive constant independent of $\epsilon$  if $\ell>  3   \epsilon^{-3}$.
	Since  $\eta_1$ and $\eta_2$ are chosen such that  the coefficients of $\eta_1 F_1,\eta_1 F_2, \eta_2 P_1,\eta_2 P_2$ are entire functions, we may assume that $N_{\eta_i}(0,r )\le N_{Q_i}(\infty,r )$, $i=1,2$, for some $Q_i\in A=\mathbb C[\lambda,\lambda^{-1}, \lambda',\frac {\beta'}{\beta}]$.
	Then for $i=1,2$, we have
	\begin{align}\label{eta0}
		N_{\eta_i}(0,r )\le  N_{Q_i}(\infty,r )\le T_{Q_i}(r) \le c_1   \epsilon^3  T_{\mathbf g}(r) 
	\end{align}
	for some constant $c_1$ independent of $\epsilon$.
	Then we derive from \eqref{gcd_GB3} that
	\begin{equation*} 
		N_{\gcd} (\tilde B(\beta_0,\beta_1), D_{(1,\beta)}(\tilde B)(\beta_0,\beta_1),r)\le_{\exc} c_2 \epsilon^3  T_{\mathbf g}(r)
	\end{equation*}
	for some constant $c_2$.  Since $c_2$ is independent of $\epsilon$ and $\epsilon$ is sufficiently small, we can now conclude (i).

	To prove (ii), we first express $B(T)=\sum_{i\in I_B} b_i(\lambda) T^{i}\in \mathbb C[\lambda,\lambda^{-1}][T]$, where $b_i\ne 0$ if $i\in I_B$.
	Then 
	\begin{align}\label{expressB}
		\tilde B(\beta_0,\beta_1)- \sum_{i\in I_B}  b_i(\lambda)  \beta_0^{d_B-i}\beta_1^{i}=0,
	\end{align}
	where $d_B:=\deg B$.   If some proper sub-sum of \eqref{expressB} vanishes, then we have 
	$\sum_{i\in I_C}  b_i(\lambda)  \beta^{ i}=0,$
	for some index subset $ I_C$ of $I_B$.  
	Let $k$ be the largest integer in $I_C$.  Then $\beta^k+\sum_{i \in I_C\setminus \{k\}}  b_i(\lambda) b_k^{-1}(\lambda) \beta^i=0$. 
	Hence, by \cite[Theorem~A3.1.6]{ru2021nevanlinna} and Proposition \ref{basic_prop}, 
	$$
	T_{\beta}(r)\le \sum_{i\in I_C\setminus \{k\}}  T_{\frac{b_i(\lambda)}{ b_k (\lambda)}}(r)+O(1)\le k  T_B(r)+O(1) \le d_B T_B(r)+O(1),
	$$
	which leads to a contradiction by the same arguments for \eqref{Tgbeta3}.  Therefore, we can assume that no  proper sub-sum of \eqref{expressB} vanishes.
	Therefore, we may apply Theorem \ref{trunborel}  to \eqref{expressB} by noting that $ d_B $ and $0$ are in $I_B$ to obtained the following 
	\begin{align}\label{countingB}
		N^{(d_B)}_{ \tilde B(\beta_0,\beta_1)}(0,r)+ \sum_{i\in I_B}  N^{(d_B)}_{\beta_0^{d_B-i}\beta_1^{i}}(0,r)&\ge_{\exc}  d_B T_{\beta}(r)- \sum_{i\in I_B}N_{b_i(\lambda)}(0,r)\cr
		&\ge d_B T_{\beta}(r)-  c_3d_B^2 T_{\lambda}(r),
	\end{align}
	for some positive constant $c_3$, independent of $\epsilon$.   
	On the other hand, since the zeros of $\beta_0$ and $\beta_1$ come from the zeros of $g_i$, $0\le i\le 2$, for any nonnegative integers $A$ and $B$ and a positive integer $n$, we have 
	\begin{align}\label{beta01}
		N^{(n)}_{\beta_0^A\beta_1^B}(0,r) \le \sum_{i=0}^2 N^{(n)}_{g_i}(0,r)\le \frac{n}{\ell}\sum_{i=0}^2 N_{g_i}(0,r)\le \frac{3n}{\ell}T_{\mathbf{g}}(r).
	\end{align}
	Then by \eqref{countingB} and \eqref{beta01}, we have 
	\begin{align}\label{beta02}
		N^{(d_B)}_{ \tilde B(\beta_0,\beta_1)}(0,r)\ge_{\exc} d_B T_{\beta}(r)- c_3d_B^2 T_{\lambda}(r)- \frac{3d_B^2}{\ell}T_{\mathbf{g}}(r).
	\end{align}

	We now treat the case that $n_1< 0$. 
	We note that in this case $d_B= \max\{-n_1,n_2\}\cdot d =-n_1d $ if we assume that $\lambda\ne\delta_i$ as in \eqref{alpha}.  In other words, the image of  $\mathbf{g}$ is not contained in $[x_1-\delta_i x_2=0]$, $1\le i\le d=\deg G$.
	By \eqref{gtildeB0}, we have 
	\begin{equation}\label{relation_GB}
		N_{G(\mathbf{g})}^{(d_B)}(0,r) \ge N^{(d_B)}_{\tilde B(\beta_0,\beta_1) }(0,r)-N^{(d_B)}_{\beta_0^{  d_B}}(0,r) 
		\ge N^{(d_B)}_{\tilde B(\beta_0,\beta_1) }(0,r)-\frac{3d_B}{\ell}T_{\mathbf{g}}(r) 
	\end{equation}
	by \eqref{beta01}.
	Since $d_B= \max\{-n_1,n_2\}\cdot d=-n_1d$ in this case, from \eqref{Tgbeta1} we have
	\begin{equation*}
		d_B T_{\beta}(r)\ge d T_{\mathbf{g}}(r) - 2d |n_1| T_{\lambda}(r).
	\end{equation*}
	Thus, we can derive from \eqref{beta02} and \eqref{relation_GB} that 
	\begin{align}\label{countingB0}
		N_{G(\mathbf{g})}(0,r)\ge_{\exc}  (d-\frac{4d^2n_1^2}{\ell})\cdot  T_{\mathbf{g}}(r)- (c_3+2 )d|n_1| T_{\lambda}(r). 
	\end{align}
	Since $\epsilon$ is sufficiently small, $ |n_1|+|n_2|\le O(\epsilon^{-1})$ and $T_{\lambda}(r)\le \epsilon^3  T_{\mathbf{g}}(r)$, by choosing $\ell>O(\epsilon^{-3})$, we arrive at
	\begin{align}\label{gcdd}
		N_{G(\mathbf{g})}(0,r)\ge_{\exc}  (1-\epsilon) d \cdot T_{\mathbf{g}}(r).
		\end{align}
	Together with (i), we have $N^{(1)}_{G(\mathbf{g})}(0,r) \ge_{\exc}  (1-2\epsilon) d \cdot T_{\mathbf{g}}(r).$
	
	For the case $n_1 \ge 0$,    we have $d_B=(n_1+n_2)d$, and 
	from \eqref{gtildeB2} we have 
	\begin{align}\label{countingG0}
		N^{(d_B)}_{G(\mathbf{g}) }(0,r)&\ge N^{(d_B)}_{\tilde B(\beta_0,\beta_1) }(0,r)-N^{(d_B)}_{\beta_0^{ n_2 d}}(0,r)-N^{(d_B)}_{\beta_1^{ n_1 d}}(0,r)\cr
		&\ge N^{(d_B)}_{\tilde B(\beta_0,\beta_1) }(0,r)-\frac{6d_B}{\ell}T_{\mathbf{g}}(r) 
	\end{align}
	by \eqref{beta01}.
	The rest of the argument to derive \eqref{gcdd} which we omit is similar to the previous one.
	
	Finally, we note that the exceptional set  $W$ consists of two types as follows. The first type is $[x_1^{n_1}x_2^{n_2}=\beta x_0^{n_1+n_2}]$, if $n_1\ge 0$; 
	$[x_2^{n_2}x_0^{-n_1-n_2}=\beta  x_1^{-n_1}]$, if $n_1< 0$,
	where $\beta\in\{\alpha_1,\hdots,\alpha_t,\gamma_1,\hdots,\gamma_s\}$ is a zero  of the resultant defined in \eqref{zeroresultant}  or a  (possible) zero  of $B(\Lambda,0)$.
	The second type is of the form $[x_1-\delta_j x_2=0]$, where $\delta_j\in \mathbb C$ are defined in \eqref{alpha}  with $n_1=-1$ and $n_2=1$. 
\end{proof}

\section{Proof of Theorem \ref{orbifoldGGL}, Theorem \ref{GG_conj} and 	 Theorem \ref{finitemorphism}}\label{others}

Theorem \ref{orbifoldGGL} is a direct consequence of Theorem \ref{main_thm_1}.  
The proof of Theorem \ref{finitemorphism} is inspired by the arguments of Corvaja and Zannier in \cite{corvaja2013algebraic} for function fields.

\subsection{Proof of Theorem \ref{orbifoldGGL}} 
\begin{proof}[Proof of Theorem \ref{orbifoldGGL}]
	After a linear change of variables, we may assume that the $H_1, H_2, H_{3}$ are the coordinate hyperplanes of $\mathbb P^2$.  
	We note that $\deg \Delta>3$ implies $\Delta_0$ is not trivial.  Therefore, 
	we may express $\Delta_0=(1-\frac 1{n_1})D_1+\hdots+(1-\frac 1{n_q})D_q$, where $D_1,\hdots,D_q$  be distinct irreducible curves  in $\mathbb P^2(\mathbb C)$ and $n_i\in(1,\infty]\cap\mathbb Q$.    For a non-constant orbifold entire curve  ${\mathbf f}  : \mathbb C \to  (\mathbb P^2,\Delta)$, we have ${\mathbf f} (\mathbb C) \not\subset  |\Delta |$ and
	${\rm mult}_t({\mathbf f} ^*D_i) \ge 2$ for all $1\le i\le q$ and all $t \in\mathbb C$ with ${\mathbf f} (t) \in D_i$ and ${\rm mult}_t({\mathbf f} ^*H_j) \ge m_j$ for all $1\le j\le 3$ and all $t \in\mathbb C$ with ${\mathbf f} (t) \in H_j$. 
	Let ${\mathbf f} =(f_0,f_1,f_2)$ be a reduced form. 
	Then the zero multiplicity of $f_i$, $0\le i\le 2$, is at least $m_i$ if it is not zero.
	Let $D_i= [G_i=0]$, $1\le i\le q$, where $G_i\in \mathbb C[x_0,x_1,x_2]$ is irreducible.  Let $G=G_1\cdots G_q$.
	Then the zero multiplicity of $G({\mathbf f} ):=G(f_0,f_1,f_2) $ at any $z_0\in\mathbb C$ is either  zero or at  least 2.
	Hence,
	\begin{align}\label{mutiG}
		N^{(1)}_{G({\mathbf f} )}(0,r)\le \frac12 N_{G({\mathbf f} )}(0,r)\le \frac12 \deg G\cdot T_{{\mathbf f} }(r)+O(1).
	\end{align}
	To apply Theorem \ref{main_thm_1}, we let $0<\epsilon<\frac13$.
	Then  there exists a proper Zariski closed subset $W$ and  a positive integer $\ell$   independent of  ${\mathbf f} $
	such that 	 if $m_i\ge\ell$ and the image of ${\mathbf f} $ is not contained in $W$,    we have  
	$$
	N^{(1)}_{G({\mathbf f} )}(0,r)\ge_{\rm exc}  (\deg  G-\epsilon)\cdot T_{{\mathbf f} }(r).
	$$
	Together with \eqref{mutiG}, it yields
	$$
	\frac12 \deg G\cdot T_{{\mathbf f} }(r)\le_{\rm exc}  \epsilon\cdot T_{{\mathbf f} }(r)+O(1),
	$$
	which is not possible since $\epsilon<\frac13$.
	This shows that the image of ${\mathbf f} $ is contained in $W$.
\end{proof}

\subsection{Proof of Theorem \ref{GG_conj}}
\begin{proof}[Proof of Theorem \ref{GG_conj}]
	Denote by $d_i$   the degree of the  irreducible homogeneous polynomial $F_i\in \mathbb C[x_0,x_1,x_2]$, $D_i:=[F_i=0]$ for $1\le i\le 3$, and 
	$\Delta= (1-\frac 1{m_1})D_1+ (1-\frac 1{m_2})D_2+(1-\frac 1{m_{3}})D_{3}$.  The condition that $\deg \Delta>3$ implies that  $\sum_{i=1}^{3}\deg F_i\ge 4$. 
	Since the hypersurfaces  $D_1$, $D_2$,  $D_3$ intersect transversally, they  do not have a common zero and the association $P\mapsto[F_{1}^{a_{1}}(P):F_{2}^{a_{2}}(P):F_{3}^{a_{3}}(P)]$,
	where $a_{i}\coloneqq{\rm lcm}(d_{1},d_{2},d_{3})/d_{i}$,
	defines a finite morphism $\pi:\PP^{2}(\mathbb C) \to\PP^{2}(\mathbb C)$.

	It is well-known that the ramification divisor of $\pi$ is  the zero locus of the determinant $J\in \mathbb C[x_0,x_1,x_2]$  
	of the Jacobian matrix 
	$$
	\big(\frac{\partial F_i^{a_i}}{\partial x_j}\big)_{1\le i\le 3, 0\le j\le 2}
	$$ 
	of $\pi$.
	Our plan is to show that there exists an irreducible factor $\tilde G$ of $J$ in $\mathbb C[x_0,x_1,x_2]$ such that the corresponding  hypersurfaces  of $\tilde G$, $D_1$, $D_2$,  $D_3$ are in  general position.
	Furthermore, we will show that  $\tilde G(\mathbf{f})$ has very few zeros and hence conclude that the image of $\mathbf{f}$ is contained in a hypersurface of bounded degree in $\mathbb P^2(\mathbb C)$ by applying Theorem \ref{SMTmoving}  for  the hypersurface defined by  $\tilde G$, $F_{1},F_2,F_{3}$.
	
	Observing that $J$ has a factor $G\in \mathbb C[x_0,x_1,x_2]$ which  denotes  the determinant of  
	\begin{align*}
		M:=\big(\frac{\partial F_i}{\partial x_j}\big)_{1\le i\le 3, 0\le j\le 2}.
	\end{align*} 
	We note that  $G$ is not a constant 
	since each $F_i$ is homogeneous and irreducible and $\sum_{i=1}^{3}\deg F_i\ge 4$.   We claim that $[G =0]$, $D_1$, $D_2$,  $D_3$ are  in general position. To prove this, it suffices to show that $G$ does not vanish at any intersection point of any 2 divisors among $D_1$, $D_2$,  $D_3$.  By rearranging the indices, it suffices to consider that $P\in D_1 \cap D_2$ and show that $G(P)\ne0$.  Since $D_1$, $D_2$,  $D_3$ have no common zeros,  we see that $F_{3} (P)\ne 0$. 
	Using the Euler formula 
	$$
	\sum_{j=1}^{3}\frac{\partial F_i}{\partial x_j}x_j=d_i \cdot F_i, 
	$$
	we obtain 
	\begin{equation*}
		\begin{split}
			x_0G =\det \begin{pmatrix}
				d_1 F_1& \frac{\partial F_1}{\partial x_1}   & \frac{\partial F_1}{\partial x_2} \\
				d_2F_2&  \frac{\partial F_2}{\partial x_1}   & \frac{\partial F_2}{\partial x_2}\\
				d_{3}F_{3} &   \frac{\partial F_3}{\partial x_1}   & \frac{\partial F_3}{\partial x_2}
			\end{pmatrix},
		\end{split}
	\end{equation*}
	and hence
	$$
	x_0(P)G(P)= d_{3}F_{3}(P)\det \left( \frac{\partial F_i}{\partial x_j}(P) \right)_{1\le i, j\le 2}.
	$$
	Since $D_1$, $D_2$,  $D_3$ intersect transversally, we see that $\det \left( \frac{\partial F_i}{\partial x_j}(P) \right)_{1\le i, j\le 2}\ne 0$.  Then $G(P)\ne 0$ as  $F_{3}(P)\ne 0$.  This proves our claim. Hence, there is a nonconstant irreducible factor $\tilde G$ of $G$ (and hence of $J$)  in $\mathbb C[x_0,x_1,x_2]$ such that $Z:=[\tilde G=0]$,  $D_{1},D_2,D_{3}$
	are in general position.
	
	Since $\pi$ is a finite morphism and $\tilde G$ is irreducible, $\pi(Z)$ is the zero locus of an irreducible homogeneous polynomial $A\in \mathbb C[y_0,y_1,y_2]$ and the vanishing order of $ \pi ^*A$ along $Z$ is at least 2.   Then this construction gives
	$\pi^*\circ   A=\tilde G^2H$ for some $H\in  \mathbb C[x_0,x_1,x_2]$. 
	Next, we verify that $\pi(Z)=[A=0]$,$ [y_0=0],  [y_1=0],  [y_2=0]$ are in general position.
	It suffices to show that none of the points $[0:0:1], [0:1:0], [1:0:0]$ is in $\pi(Z)$.
	If any of the points,  say $[0:0:1]\in \pi(Z)$, then there exists $P\in Z$ such that $F_1(P)=F_2(P)=0$, which is impossible since $Z$, $D_{1},D_2,D_{3}$ are in general position.   
	
	Now let $\mathbf{f}=(f_{0},f_{1},f_{2}):\CC\to\PP^{2}$
	be a holomorphic map, where $f_{0},f_{1},f_{2}$ are entire functions
	without common zeros, such that 
	\begin{align}\label{umap}
		\mathbf{u}\coloneqq\pi(\mathbf{f})=(F_{1}(\mathbf{f})^{a_{1}},F_{2}(\mathbf{f})^{a_{2}},F_{3}(\mathbf{f})^{a_{3}})
	\end{align}
	is a 3-tuple of entire functions with zero multiplicity at least $a_i\cdot m_i$ for the $i$-th position, $1\le i\le 3$. From the equality $ A(\mathbf{u})=(\pi^*\circ   A)(\mathbf{f})=\tilde G^2(\mathbf{f})H(\mathbf{f})$, it follows that for each $z\in\mathbb{C}$ with $v_{z}(\tilde G(\mathbf{f}))>0$, 
	we have 
	\begin{align}\label{poleH}
		v_{z}( A(\mathbf{u}))\ge 2v_{z}(\tilde G(\mathbf{f})) \ge v_{z}(\tilde G(\mathbf{f}))+1
	\end{align}
	as  $f_{0},f_{1},f_{2}$ are entire functions.  Therefore,
	\[
	N_{\tilde G(\mathbf{f})}(0,r)\le N_{A(\mathbf{u})}(0,r)-N_{A(\mathbf{u})}^{(1)}(0,r).
	\]
	Then we may apply Theorem \ref{main_thm_1}  with the nonconstant polynomial  $A\in \mathbb C[y_0,y_1,y_2]$ as it is irreducible and the zero locus is in general position with the coordinate lines.  Then for a given $\epsilon >0$, there exists a proper Zariski closed subset $W\subset \mathbb P^2$ and a (sufficiently large) positive integer $\ell_1$ such that 
	if $m_i\ge \ell_1$ for $1\le i\le 3$ and  the image of the holomorphic map $\mathbf{u} $ as in \eqref{umap} is not contained in $W$, then 
	$$
	N_{A(\mathbf{u})}(0,r)-N_{A(\mathbf{u})}^{(1)}(0,r)\le_{\exc} \frac{\epsilon}{d_1a_1} T_{\mathbf{u}}(r)=\epsilon T_{\mathbf{f}}(r).
	$$
	Therefore,
	\begin{align}\label{zeroG}
		N_{\tilde G(\mathbf{f})}(0,r)\le_{\exc} \epsilon T_{\mathbf{f}}(r) 
	\end{align}
	if the image of $\mathbf{f}$ is not contained in $\pi^{-1}(W)$.
	
	Finally, since  $[\tilde G=0]$,  $D_{1},D_2$ and $D_{3}$
	are in general position, Theorem \ref{SMTmoving}
	implies that for any    $0<\epsilon<\frac 14$  there exist two positive integers $M$ and $N$ (independent of  $\mathbf{f}$) such that 
	\begin{align}\label{applySMT1}
		\left(1-\epsilon\right)T_{\mathbf{f}}(r)&\le_{\rm exc} \frac{1}{\deg \tilde G}N_{\tilde G(\mathbf{f})}(0,r)+\sum_{j=1}^{3}\frac{1}{\deg F_{j}}N^{(M)}_{F_{j}(\mathbf{f})}(0,r), 
	\end{align}
	or the image of $\mathbf{f}$ is contained in a  plane curves with degree bounded by $N$.
	Let $m_i\ge \ell_2:=   3 M\epsilon^{-1}$.  Then 
	\begin{align}
		\frac{1}{\deg F_{j}}N^{(M)}_{F_{j}(\mathbf{f})}(0,r) \le  \frac{1}{\deg F_{j}}\frac{M}{\ell_2 } N_{F_{j}(\mathbf{f})}(0,r)\le  \frac{\epsilon}{3} \cdot T_{\mathbf{f}}(r).
	\end{align}
	Together with \eqref{zeroG}, we derive from \eqref{applySMT1}  that 
	\begin{align}\label{applySMT3}
		\left(1-\epsilon\right)T_{\mathbf{f}}(r)&\le_{\rm exc}  2 \epsilon\cdot  T_{\mathbf{f}}(r),
	\end{align}
	which is not possible since  $\epsilon<\frac 14$.
	Therefore, we conclude that if $m_i\ge \ell:=\max\{\ell_1,\ell_2\}$ for $1\le i\le 3$, then the image of $\mathbf{f}$ is contained in some plane curve of degree bounded  by $N$, where $N$ is independent of $\mathbf{f}$.
\end{proof}

\subsection{Proof of Theorem \ref{finitemorphism}}\label{ProofTheorem3}

\begin{proof}[Proof of Theorem \ref{finitemorphism}]
	Let $\pi:   X\to \mathbb P^2$  be a   finite morphism.
	Let $H_i=[x_{i}=0]$, $0\le i\le 2$,  
	$D_i$ be the support of $\pi^* H_i$  and   $\Delta= (1-\frac 1{m_0})D_0+(1-\frac 1{m_1})D_1+(1-\frac 1{m_{2}})D_{2}$.
	Let $\mathbf{f} : \mathbb C \to  ( X,\Delta) $  be a non-constant orbifold entire curve, i.e.
	$f(\mathbb C) \not\subset  |\Delta |$
	and 
	\begin{align}\label{multiplicity1}
		m_i\le {\rm mult}_t(\mathbf{f}^* {E}_i)=   {\rm mult}_t((\pi\circ \mathbf{f})^* H_i) 
	\end{align}
	for  $0\le i\le 2$ and all $t \in\mathbb C$ with $\mathbf{f}(t) \in {E}_i$,  for any component $E_i$ of $D_i$.
	Let $(f_0,f_1,f_2)$ be a reduced representation of  $\pi\circ \mathbf{f} :\mathbb C\to\mathbb P^2$, i.e. $\pi\circ \mathbf{f} =(f_0,f_1,f_2)$ and $f_0,f_1,$ and $f_2$ are entire functions with no common zeros.
	Then \eqref{multiplicity1} implies that for  $0\le i\le 2$ 
	\begin{align}\label{multiplicity2}
		{\rm mult}_t(f_i)\ge m_i \quad\text{for all $t \in\mathbb C$ with $f_i(t)=0$}.
	\end{align}

	We now recall some arguments from \cite[Lemma 1]{corvaja2013algebraic}. 
	By \cite[(1.11)]{debarre2001-higher}, the canonical divisor class  $K_{ X}$ on $ X$ can be written as 
	$ K_{ X}\sim  \pi^*(K_{\PP^2})+{\rm Ram}$, where ${\rm Ram}$ is the ramification divisor of $\pi$.
	Let ${\rm Ram}=Z+R_D$, where $R_D$ is the contribution coming from the support contained in $D$, i.e.
	$\pi^*H_1+\pi^*H_2+ \pi^*H_{3}=D+R_D$. 
	Since ${K}_{\PP^2}\sim -(H_1+H_2+H_{3})$,  we obtain 
	\begin{equation}\label{ramifidivisor}
		Z\sim  D_1+D_2+D_3+ K_{ X}. 
	\end{equation}
	Since  $(X,\Delta)$ is   of  general type, $K_X+\Delta$ is big and hence $Z$ is big as well.

	As $ \pi(Z)$ is a curve in $\mathbb P^2$,  it is the zero locus of   a   homogeneous polynomial $F \in \mathbb[x_0,x_1,x_2]$.  
	We note that $\pi(Z)=[F=0]$  and the coordinate hyperplanes $[x_i=0]$, $0\le i\le 2$ are in general position  by
	the assumption that $\pi(Z)$ does not intersect the set of points  $\{(1,0, 0),(0,1, 0),(0,0,1)\} $ in $\mathbb P^2$. 
	
	Let $Z_0$ be an irreducible component of $Z$ and  
	$F_0 $ be the irreducible factor of $F$ in $\mathbb C[x_0,x_1,x_2]$ such that its zero locus  $R_0:=[F_0=0]=\pi(Z_0)$.  Then $ \pi^*R_0$ has multiplicity at least 2 along $Z_0 $.  
	Moreover,   the vanishing order of $\mathbf{f}$ along $ \pi^*R_0$ equals the vanishing order of $\pi\circ \mathbf{f}=(f_0,f_1,f_2)$ along $R_0$.
	Let $g_U$ be a local defining function of $Z_0$ in an open set $U$ of a point $x\in Z_0$.  Then 
	\begin{align}\label{twice}
		{\rm ord}_t (F_0(f_0,f_1,f_2))\ge 2{\rm ord}_t(g_U\circ \mathbf{f}) 
	\end{align} 
	for $t\in\mathbb C$ such that $\mathbf{f}(t)=x$. 
	Therefore, the zero multiplicity of $F_0(f_0,f_1,f_2)$ at $t$ is at least twice of ${\rm mult}_t(\mathbf{f}^* Z_0)$.
	Therefore,
	\begin{align}\label{coutingzero1}
		N_{\mathbf{f}}( Z_0,r )\le N_{F_0(f_0,f_1,f_2)}(0,r)-N^{(1)}_{F_0(f_0,f_1,f_2)}(0,r).
	\end{align}
	
	We are now in position to apply Theorem \ref{main_thm_1}  for $F_0$.   Then for any 
	$\epsilon >0$, there exists a proper Zariski closed subset $W$ of $\mathbb P^2$ and  a positive integer $\ell_1$   independent of  $\mathbf{f}$
	such that 	 if $m_i\ge\ell_1$ for $i=0,1,2$ and the image of $\pi\circ \mathbf{f}$ is not contained in $W$, then 
	$N_{F_0(f_0,f_1,f_2)}(0,r)-N^{(1)}_{F_0(f_0,f_1,f_2)}(0,r) \le_{\exc} \epsilon T_{\pi\circ \mathbf{f}}(r)$.
	Therefore,   
	\begin{align}\label{zeroZ}
		N_{\mathbf{f}}( Z_0,r )\le_{\exc} \epsilon T_{\pi\circ\mathbf{f}}(r) 
	\end{align}
	if $m_i\ge\ell_1$ for $i=0,1,2$ and the image of $\pi\circ \mathbf{f}$ is not contained in $W$.

	On the other hand,  since  $[F_0=0]$,  $H_1,H_2$, and $H_3$
	are in general position, Theorem \ref{SMTmoving} implies that
	for any    $0<\epsilon<\frac 14$  there exist two positive integers $M$ and $N$ (independent of  $\mathbf{f}$) such that 
	either the image of $\pi\circ\mathbf{f}$ is contained in a curve in $\PP^2$	with degree bounded by $N$, or 
	\begin{align}\label{applySMT4}
		\left(1-\epsilon\right)T_{\pi\circ\mathbf{f}}(r)&\le_{\rm exc} \frac{1}{\deg F_0}N_{F_0(f_0,f_1,f_2)}(0,r)+\sum_{j=0}^{2} N^{(M)}_{f_{j}}(0,r)\cr
		&\le_{\rm exc} \frac{1}{\deg F_0}N_{F_0(f_0,f_1,f_2)}(0,r)+\sum_{j=0}^{2} \frac{M}{m_j}N_{f_{j}}(0,r)\cr
		&\le  \frac{1}{\deg F_0}N_{F_0(f_0,f_1,f_2)}(0,r)+ \epsilon T_{\pi\circ\mathbf{f}}(r)+O(1), 
	\end{align}
	if  $\min\{m_0,m_1,m_2\}> 3M{\epsilon}^{-1}$.

	By repeating the  above arguments  for each component of $Z$ and replacing $\epsilon$, $\ell_1$, $W$, $M$ and $N$ if necessary, then
	\eqref{applySMT4} remains valid by replacing $Z_0$ with $Z$ if the image of $ \mathbf{f}$ is not contained in $\pi^{-1}(W)$. 
	Hence we have
	\begin{align}\label{applySMT2}
		N_{F (f_0,f_1,f_2)}(0,r)\ge_{\exc}   \left(1-2\epsilon\right)\deg F \cdot T_{\pi\circ \mathbf{f}}(r) 
	\end{align}
	if $\min\{m_0,m_1,m_2\}\ge 3M{\epsilon}^{-1}$ and  the image of $\pi\circ\mathbf{f}$ is not contained in a curve in $\PP^2$	with degree bounded by $N$.
	We can derive from \eqref{applySMT2} that
	\begin{align}\label{Fproxi}
		m_{F(f_0,f_1,f_2)}(0,r)\le_{\exc}    2\epsilon  \deg F\cdot T_{\pi\circ\mathbf{f}}(r)+O(1).
	\end{align}   
	Then the functorial property,  $ \tilde { Z}\le \tilde \pi^* ([F=0])$ (as divisors)  implies that  
	\begin{align}\label{Fproxi2}
		m_{\mathbf{f}}(Z,r) \le m_{\pi\circ \mathbf{f}}  ([F=0]),r)  +O(1) \le 2 \epsilon  \deg F\cdot T_{\pi\circ \mathbf{f}}(r). 
	\end{align}
	Together with \eqref{zeroZ} for $Z$, we have
	\begin{align}\label{charfun}
		T_Z(\mathbf{f},r) \le_{\exc}    (2\deg F+1)\epsilon   \cdot T_{\pi\circ \mathbf{f}}(r) 
	\end{align}
	if  the image of $\pi\circ\mathbf{f}$ is not contained in a curve in $\PP^2$	with degree bounded by $ \max\{N,\deg W\}$.

	Let $A$ be an ample divisor on $ X$. 
	Then by \cite[Proposition 10.7]{vojta2009diophantine}, there exists a constant $c $  such that 
	\begin{equation}\label{bdd_ample}
		T_{\pi\circ \mathbf{f}}(r) =\frac 1{\deg F}  \cdot T_{ \pi^* ([F=0])}(\mathbf{f},r)+O(1)\le  c T_A(\mathbf{f},r)+O(1).
	\end{equation}
	On the other hand, since   $Z$ is big, 
	there exists  a  constant $b>0$ and a   proper Zariski-closed set   $W_0$ of $X$, depending only on $A$ and $Z$, such that 
	\begin{equation}\label{bdd_big}
		T_A(\mathbf{f},r)\le b  T_Z(\mathbf{f},r)+O(1),
	\end{equation}
	if the image of $f$ is not contained in $W_0$.
	Combining  this with  \eqref{charfun} and \eqref{bdd_ample}, it yields
	$$
	T_A(\mathbf{f},r)\le_{\exc} bc(2\deg F+1)\epsilon T_A(\mathbf{f},r)+O(1),
	$$
	which is not possible  as $b$ and $c$ are independent of   $\epsilon $ and $\epsilon $ can be taken sufficiently small. 
	In conclusion, 
	the image of $\pi\circ\mathbf{f}$ is contained in a curve in $\PP^2$	with degree bounded by $ N_1:=\max\{N,\deg W,\deg \pi(W_0) \}$.
	Since $\pi:X\to \mathbb P^2$ is a finite morphism, it implies that the image of $\mathbf{f}$   is contained in a curve of degree bounded by $N_1\cdot\deg \pi$, which is independent of $\mathbf{f}$.
\end{proof}

\subsection{Remark on Strong Green-Griffiths-Lang conjecture}\label{remarks}
We will discuss the exceptional sets for  Theorem \ref{GG_conj} and  Theorem \ref{finitemorphism} under the assumption that the multiplicities $m_i=\infty$, $1\le i\le 3$, i.e. the open case of the Green-Griffiths-Lang conjecture.

The proofs of  Theorem \ref{GG_conj} and  Theorem \ref{finitemorphism} are based on  Theorem \ref{main_thm_1} and Theorem \ref{SMTmoving}.  The exceptional set $W$ in Theorem \ref{main_thm_1} can be constructed explicitly.  So, the main point is to   find an alternative for Theorem \ref{SMTmoving}.   When $m_i=\infty$ for $1\le i\le 3$, we consider units instead of entire functions with sufficiently large multiplicities.  Therefore, we can replace Theorem \ref{SMTmoving} with the original theorem of Ru, where he basically considered counting functions without truncation.
We also note that the proof of Ru's theorem is an application of Cartan's second main theorem, which is under the assumption that the entire curves are linearly nondegenerate.  In \cite{vojta1997cartan}, Vojta has weaken the linearly nondegenerate condition to  that there exists $\mathcal H$, a finite union of  proper linear subspaces, such that the non-constant entire curves are not contained in  $\mathcal H$. 
Combining Vojta's refinement of the second main theorem, we can reformulate the result of Ru as follows.

\begin{theorem}[\cite{Ru2004}]\label{SMTRu}
	
	Let $\mathbf{f}$ be a nonconstant holomorphic map of $\CC$ into $\PP^n$.     Let $\{D_j\}$, $1\le i\le q$, be  hypersurfaces in $\PP^n(\mathbb C)$  of degree $d_i$, in general position.   Then for any $\epsilon>0$, there exist a hypersurface $Z$ in $\PP^n(\mathbb C)$ such that  the following inequality holds:
	\begin{equation*} 
		(q-n-1-\epsilon)T_{\mathbf{f}}(r)\le_{\exc} \sum_{j=1}^q\frac{1}{d_j} N_{Q_j(\mathbf{f})} (0,r),
	\end{equation*}
	if the image of $\mathbf{f}$ is contained in $Z$.	
\end{theorem}

Then it is clear that the use of Theorem \ref{SMTRu} allows us to find   exceptional set $W$ such that any non-constant entire curve $f$ in Theorem \ref{GG_conj} (resp.  Theorem \ref{finitemorphism}) is contained in $W$ when $m_i=\infty$ for $1\le i\le 3$.

	

\end{document}